\documentclass[11pt,reqno]{amsart}
\usepackage{fullpage}
\usepackage[T1]{fontenc}                                   %imposta la codifica dei font
\usepackage{amsfonts}
\usepackage[utf8]{inputenc}                               %lettere accentate da tastiera
\usepackage{comment}                                       %pacchetto commenti
\usepackage{mparhack}                                      %evita errori di impaginazione nelle note a margine
\usepackage{amsmath,amssymb,amsthm,mathrsfs,eucal}                      %matematica
\usepackage{booktabs}                                                                          %tabelle
\usepackage{graphicx,subfig}                                                                %figure
\usepackage{wrapfig}                                                                            %figure e tabelle nel testo
\usepackage[bookmarks=true,colorlinks=true]{hyperref}                      %per avere i segnalibri
\usepackage{bm}                                                                                   %grassetto

\newtheorem{defn}{Definition}[section]
\newtheorem{thm}{Theorem}[section]
\newtheorem{prop}{Proposition}[section]

\newtheorem{rem}{Remark}[section]

\DeclareMathOperator*{\dive}{div}

\newcommand{\R}{\mathbb{R}}

\def\eps{\varepsilon}

\def\XXint#1#2#3{{\setbox0=\hbox{$#1{#2#3}{\int}$}
\vcenter{\vspace{-1pt}\hbox{$#2#3$}}\kern-.5\wd0}}
\def\Xint#1{\mathchoice {\XXint\displaystyle\textstyle{#1}}{\XXint\textstyle\scriptstyle{#1}}{\XXint\scriptstyle\scriptscriptstyle{#1}}{\XXint\scriptscriptstyle\scriptscriptstyle{#1}}\!\int}
\def\intmed{\hbox{\ }\Xint{\hbox{\vrule height -0pt width 10pt depth 1pt}}}

\begin{document}
\author{S. Fagioli \and E. Radici}
\address{ Simone Fagioli, Emanuela Radici - DISIM - Department of Information Engineering, Computer Science and Mathematics, University of L'Aquila, Via Vetoio 1 (Coppito)
67100 L'Aquila (AQ) - Italy}
\email{simone.fagioli@univaq.it}
\email{emanuela.radici@univaq.it}
\title[Diffusion with particles]{Solutions to aggregation-diffusion equations with nonlinear mobility constructed via a deterministic particle approximation}
\date{}
\begin{abstract}
\noindent
We investigate the existence of weak type solutions for a class of aggregation-diffusion PDEs with nonlinear mobility obtained as large particle limit of a suitable nonlocal version of the follow-the-leader scheme, which is interpreted as the discrete Lagrangian approximation of the target continuity equation. We restrict the analysis to bounded, nonnegative initial data with bounded variation and away from vacuum, supported in a closed interval with zero-velocity boundary conditions. The main novelties of this work concern the presence of a nonlinear mobility term and the non strict monotonicity of the diffusion function. As a consequence, our result applies also to strongly degenerate diffusion equations. The results are complemented with some numerical simulations. 
\end{abstract}
\maketitle

\section{Introduction}
A variety of models in mathematical biology such as chemotaxis, animal swarming, pedestrian movements etc.  concern with aggregation and diffusion phenomena. A large number of works have focused on this type of description, see \cite{capasso,okubo,hughes,keller_segel,mogilner,OkLe} for a classical and incomplete list of references. One of the relevant features of biological models is the possibility to describe phenomena at two different scales: microscopic (individual based) and macroscopic (locally averaged quantities).

Given a certain biological population composed by $N$ agents located in the positions $x_1,...,x_N$ in $\R^d$, the aggregation-diffusion process acts as a drift $V$ on the agents. The microscopic approach focuses on evolution in time of the positions and in most of the applications this evolution is of first order type since inertial effects can be neglected. The resulting system of ordinary differential equations possesses as continuum (macroscopic) counterpart the continuity equation 
\begin{equation}\label{eq:cont_eq}
\partial_t \rho + \dive \left(\rho V\right)=0,
\end{equation}
where $\rho$ is the averaged density of particles and $V$ describes the macroscopic law for the velocity. One of the mathematical problems arising in the modeling is the rigorous justification of the micro-macro equivalence, namely the construction of \eqref{eq:cont_eq} via a \emph{many particle limit}.

In this work we study the convergence of a deterministic particles approximation towards weak solutions of the following one-dimensional aggregation-diffusion equation

\begin{equation}\label{eq:general}
\partial_t\rho = \partial_x \left(\mathcal{M}(\rho)\partial_x\left(a(\rho)+K\ast\rho\right)\right),
\end{equation}
posed in a closed interval $[0,\ell] \subset \R$ equipped with no-slip (zero velocity) boundary conditions, see \eqref{main} below for a precise statement. Diffusion processes are due to short-range repulsions between particles and are modeled by the function $a$, that is in general a nonlinear function of the density, while long-range attraction phenomena are modeled by the \emph{attractive} interaction kernel $K$. Typical examples of $K$ are the \emph{power laws} $K(x)=|x|^\alpha$, for $\alpha>1$, the $Gaussian$ kernels $K(x)=-C e^{-|x|^2/l}$ and the \emph{Morse} kernels $K(x)=-C e^{-|x|/l}$, see \cite{mogilner}.

We consider the nonlinear mobility term $\mathcal{M}(\rho)$ depending on the density only, and in the form
\[ \mathcal{M}(\rho)=\rho v(\rho).\]
This choice can be made more general, for instance considering mobilities that depend on the velocity gradient, see \cite{verb}. Assuming that the velocity map $v(\rho)$ is degenerate for a certain value $\rho_{max}$, the above expression of $\mathcal{M}(\rho)$ is expected to prevent the overcrowding effect. This assumption can be applied to extend classical chemotaxis models ( take as reference example the Keller-Segel model  \cite{keller_segel}) that may produce blow-up in finite time, see \cite{bel,blanchet,JL,TW} and their references, to a more $realistic$ behavior in which this phenomenon is prevented: agents stop aggregating after a certain maximal density $\rho_{max}$ is reached, see \cite{BerDiFDol,bruna_chapman,painter_hillen} and the references therein.

The particle approximation of linear and nonlinear diffusion equations dates back to the pioneering works by \cite{presutti,varadhan2,varadhan1} involving probabilistic methods. A deterministic approach was introduced in \cite{Budd,DegMul,MacCSoc,Russo1,Russo2}, mainly with numerical purpose, and gained a lot of attentions in recent years after the equation was formulated as gradient flow of a proper energy functional in the space of probability measures, see \cite{AGS,otto,S,V2}. If one considers the linear mobility $\mathcal{M}(\rho)=\rho$,  zero nonlocal interaction $K=0$ and the energy functional
\[
 \mathcal{F}\left[\rho\right]=\int_{\Omega}A(\rho)dx,  \mbox{ where }A(\rho)=\int_{0}^{\rho}a(\zeta)d\zeta,
\]
it is possible to rewrite, formally at this stage, \eqref{eq:general} as
\[
 \partial_t\rho = \partial_x \left(\rho\partial_x\left(\frac{\delta \mathcal{F}}{\delta\rho}\left[\rho\right]\right)\right).
\]
A key tool to prove this is the so-called \emph{JKO-scheme} \cite{JKO}, that consists in constructing recursively a time-approximation for the density via a variational interpretation of the Implicit Euler scheme.

% fix $\tau>0$ and define
% \begin{equation}\label{eq:JKO}
% \begin{cases}
%  \rho_{\tau}^0=\rho_0,\\
%  \rho_{\tau}^{n+1}=\argmin_{\rho\in\mP}	\left\{\frac{1}{2\tau}d_{W^2}^2(\rho_{\tau}^{n},\rho)+\mathcal{F}\left[\rho\right]\right\}
% \end{cases}
% \end{equation}
% where $\rho_0\in\mP$ is a given initial datum and $d_{W^2}$ is the $2-$Wasserstein distance, see Section 2 for a precise definition. 
This gradient flow formulation carries out a naturally induced Lagrangian description, the so called \emph{pseudo-inverse} formalism, see \cite{CT}. 
The first use of this approach dates back to a couple of papers by Gosse e Toscani \cite{GT1,GT2}, where the authors introduce a time-space discretization of the pseudo-inverse equation and study several numerical properties of this scheme. One of the reasons why the interest of the scientific community is growing more and more on this topic is, indeed, the feasibility to design powerful numerical scheme, see \cite{CG,CDMM,CHPW,CarMol,CPSW,JMO}.
Most of these works deal with the discretization of the JKO-scheme and convergence of the particle method via variational techniques.
In \cite{MatOsb} Matthes and Osberger show the convergence of the Lagrangian approximation to a weak solution of the diffusion equation posed in a bounded interval with zero velocity boundary conditions and density away from zero (no vacuum). The key ingredients in the proof are the min-max principle and the BV estimates on the discretized density.  The diffusion they consider falls in the classical set of assumptions for nonlinear (possibly degenerate) diffusion equations: defining 
\[
 P(\rho)=\int_0^\rho \zeta a'(\zeta)d\zeta,
\] 
with $P'(0)=0$, $P'(\rho)>0$ for $\rho>0$, and under suitable regularity assumptions,  the nonlinear diffusion equation, i.e. \eqref{eq:general} with $K=0$, can be rephrased in the \emph{Laplacian form} $\partial_t\rho = \partial_{xx}P(\rho)$. A reference example is the porous medium equation \cite{vazquez}, where $P(\rho)=\rho^m$ and $a(\rho)=m(m-1)^{-1}\rho^{m-1}$ for some $m>1$. 

In a later work \cite{MatSol}, Matthes and S\"ollner tackle the problem of particle approximation for aggregation-diffusion through the discretization of the JKO. In that paper they prove convergence of the approximating particle sequence to the weak solutions of the pseudo-inverse equation.

One of the novelties of our work is the presence of the nonlinear mobility $\mathcal{M}(\rho)$. This term complicates the problem because it prevents the construction of the approximating sequence via discretization of the JKO-scheme, as it is done instead in most of the papers quoted above. The reason is that, in this case, the Wasserstein distance is not explicit but it can only be expressed in the Benamou-Brenier formulation, see \cite{CCN}, even if the equation preserves a gradient flow structure, see \cite{DF}. However, in presence of the nonlinear mobility, it is easy to reformulate the diffusion equation in Laplacian form by calling
\begin{equation}\label{eq:diffusione}
 \phi(\rho)=\int_0^\rho \zeta v(\zeta) a'(\zeta)d\zeta.
\end{equation}
The main difference between $\phi$ and $P$ is that $\phi'$ vanishes together with $v$ for some values $\rho>0$, in contrast with the classical porous medium type equations where the degeneracy is allowed only in $\rho=0$. From the above considerations we can deduce that the natural assumptions on $\phi$ are 
\[
\mbox{Lipschitz regularity and non-decreasing monotonicity,}
\]
see hypothesis (Diff) below for a precise statement. Included in this class of diffusions are the classical porous medium equations (one-point degeneracy), the two-phase reservoir flow equations (two point degeneracy), the so-called \emph{strongly} degenerate diffusion equations where $\phi'(\rho)=0$ for $\rho\in\left[\rho_1,\rho_2\right]$, see \cite{BBK,KR}.

Since the lagrangian procedure illustrated above seems not be applicable in the presence of nonlinear mobility, we tackle the problem of particle approximation of \eqref{eq:general}, among all the possible strategies, by adapting the techniques used in \cite{DFR}. This approximation approach is based on a deterministic ODEs strategy, the so-called \emph{Follow-the-Leader} model, introduced in \cite{AKMR,Whitham}. In \cite{DFFR,DFR} the authors prove convergence in the large particle limit of the \emph{Follow-the-Leader} model to local nonlinear conservation laws having in mind applications to traffic models. This technique was also extended to other traffic/pedestrian models such as the Hughes model and Aw-Rascle-Zhang model, see  \cite{DFFR2,DFFRR,DFFRR2}. The discrete-to-continuum limit is largely studied in the contest of vehicular traffic by using different approaches, we mention here the derivation from kinetic models in \cite{DeDe}, the paper \cite{belbel} where the microscopic scale is modeled by methods of game theory, and the probabilistic approach in \cite{Fer,FerNej} and their references. 

The link between equation \eqref{eq:general} and scalar conservation laws should not be surprising since, in the intervals where $\phi'$ is zero (degenerate diffusion), \eqref{eq:general} reduces to a conservation law with nonlocal flux, see \cite{BBK}. In \cite{DFFRA} the \emph{Follow-the-Leader} particle method is extended to conservation laws with nonlocal flux proving convergence of the approximation to entropy solutions of the problem.

The precise statement of the Initial-boundary value Problem we study is the following:
\begin{equation}\label{main}
\left\lbrace \begin{array}{lll} 
\partial_t \rho = \partial_{xx} \phi(\rho) + \partial_x (\rho v(\rho)K' \ast \rho) & (t,x) \in (0,T)\times [0,\ell], \\
\rho(0,x) = \bar{\rho}(x)  &(t,x) \in \{0\}\times [0,\ell], \\
v(\rho)\left(\partial_x a(\rho) + K' \ast \rho \right)= 0  &(t,x) \in (0,T)\times \{0\} \cup (0,T)\times \{\ell\},
\end{array} \right. 
\end{equation}
where $a$ and $\phi$ are coupled by \eqref{eq:diffusione}. The initial datum $\bar{\rho}$ satisfies the following assumptions:
\begin{itemize}
\item[(In1)] $\bar{\rho} \in BV([0,\ell]; \mathbb{R}^+)$ with $\int_0^{\ell} \bar{\rho}(x) dx = \sigma$, for some $\sigma>0$,
\item[(In2)] there are some fixed parameters $m,M >0$ such that $m \leq \bar{\rho}(x) \leq M$ for every $x \in [0,\ell]$.
\end{itemize}
Clearly, these conditions are not in contradiction as soon as $m \leq \frac{\sigma}{\ell} \leq M$.  
The diffusive term $\phi$, the velocity $v$ and the interaction kernel $K$ are taken under the following assumptions  
\begin{itemize}
\item[(Diff)] $\phi: [0,\infty) \to \R$ is a nondecreasing Lipshcitz function, with $\phi(0)=0$;
\item[(Vel)] $v: [0,\infty) \to \R$ is monotone decreasing and piecewise $C^1$ function such that $v(z)=0$ for every $z \geq \frac{M}{\sigma}$.
% questo significa che se v è non crescente allora è positiva, altrimenti se v è non  decrescente allora è negativa
\item[(Ker)] $K \in L^1_{loc}(\R)$ is a nonlocal attractive potential, radially symmetric: $K'(x)>0$ for every $x>0$, $K'(0)=0$ and $K'(x)<0$ for $x<0$. Moreover, we assume that $K'$ is a $L_1$-Lipschitz continuous at least on $[-2\ell,\,2\ell]$ and 
\begin{equation*}\label{boundsuK'}
\sup_{x \in [-2\ell,\,2\ell]} |K''(x)| < L_2\,,\quad \mbox{ and }\quad \sup_{x \in [-2\ell,\,2\ell]} |K'''(x)| < L_3\,,
\end{equation*}
for some positive constants $L_2, L_3$. Then we set $L:= \max\{L_1,\,L_2,L_3\}$.
\end{itemize}
The particle approximation we are dealing with can be formally derived in this way. Assume that $\rho$ has unit mass and consider the cumulative distribution function associated to $\rho$
\[
  F(t,x)=\int_{-\infty}^x \rho(t,y)dy,
\]
and its pseudo-inverse
\[
 X(t,z)=\inf\left\{F(t,x)\geq z\right\} \quad \mbox{for }z\in\left[0,1\right],
\]
see Section 2 below. A formal computation \cite{CT} shows that the pseudo-inverse function satisfies the following equation,
\begin{equation}\label{eq:psed}
 \partial_t X(t,z) = -\partial_z \left(\phi\left(\frac{1}{\partial_z X(t,z)}\right)\right)-v\left(\frac{1}{\partial_z X(t,z)}\right)\int_0^1K'\left(X(t,z)-X(t,\zeta)\right)d\zeta.
\end{equation}
Suppose now that we want to solve \eqref{eq:psed} numerically. A first attempt can be a finite difference approximation in space: consider $N\in\mathbb{N}$ and a uniform partition of the interval $\left[0,1\right]$ of size $\Delta z$ with nodes $\left\{z_i\right\}_{i=0}^{N}$. Let us denote $X_i(t)=X(t,z_i)$. Since the space step $\Delta z$ (again formally) can be considered as a mass variable the ratio
\[
 \rho_i(t)=\frac{\Delta z}{X_{i+1}(t)-X_{i}(t)},
\]
is a good candidate for approximating our density in the interval $\left[X_{i}(t),X_{i+1}(t)\right]$. Then equation \eqref{eq:psed} reduces to the following system of ODEs:
\begin{align*}
 \dot{X}_i(t) = &\frac{1}{\Delta z} \left(\phi\left(\rho_{i-i}(t)\right)-\phi\left(\rho_i(t)\right)\right)-v\left(\rho_i(t)\right)\Delta z\sum_{i=0}^{N}K'\left(X_i(t)-X_j(t,)\right).
\end{align*}

In this work we prove that a proper definition of discrete density in the spirit of the scheme formally sketched above converges to a solution of the Initial-boundary value problem \eqref{main} in a suitable weak sense under the assumptions (In1), (In2), (Diff), (Vel) and (Ker).

The paper is organized as follows. In the first part of Section 2 we define in details our particle scheme and the corresponding discrete density and we state our main result in Theorem \ref{maintheorem}, after that we recall the notion of Wasserstein distance together with some useful properties and we close the Section proving a discrete version of the min-max principle in Proposition \ref{st:MinMax}. Section 3 is entirely devoted to the proof of Theorem \ref{maintheorem}. Finally, in Section 4 we provide some numerical simulations in order to show some evidence of patterns formation that can be interesting from the modeling point of view.

\section{Particle approximation and statement of the main result}

In this section we define the particle approximation scheme formally sketched above, and we give the statement of the main Theorem \ref{maintheorem}, in terms of the density approximation defined in the discrete structure. Then we recall the definition and some basic properties of the one dimensional Wasserstein distance that will be useful to study the convergence in suitable topology of the sequence of the discrete densities and we conclude the section showing that the \emph{maximum-minimum principle} holds in our setting, thus preventing from blow-up behavior as well as vacuum zones.
Given an initial datum $\bar{\rho}: [0,\ell]\to \mathbb{R}$ satisfying (In1) and (In2), a final time $T>0$ and $N \in \mathbb{N}$, we atomize $\bar{\rho}$ in $N+1$ particles: set $x_0(0)=0$ and define recursively 
\[ x_i(0) = \sup \left\lbrace x\in \R : \int_{x_{i-1}(0)}^x \bar{\rho}(z)dz < \frac{\sigma}{N} \right\rbrace,\quad \forall\, i=1,\ldots, N-1.\]
It is easy to check that $x_N(0)=\ell$ and $x_{i+1}(0) - x_i(0) \geq \frac{\sigma}{MN}$ for every $i=0,\ldots, N-1$.
Taking the above construction as initial condition, we let the particles evolve for all $t \in [0,T]$ accordingly to the following ODEs system  
\begin{equation}\label{ODES}
\begin{cases}
\dot{x}_0(t) = 0, \\
\dot{x}_i (t) = \dot{x}^d_i(t) + \dot{x}^{nl}_i(t) \quad i=1,\ldots,N-1, \\
\dot{x}_N(t) = 0, 
\end{cases}
\end{equation}
where we have set for $i=1,\ldots,N-1$,
\begin{align}
 \dot{x}^d_i(t) &= N(\phi(R_{i-1}(t)) - \phi(R_i)(t)),\label{part_diff}\\
\dot{x}^{nl}_i(t)&= -\frac{v(R_i(t))}{N} \sum_{j>i}K'(x_i(t) - x_j(t)) - \frac{v(R_{i-1}(t))}{N}\sum_{j<i} K'(x_i(t) - x_j(t)).\label{part_nonl}
\end{align}
and 
\begin{equation*}
R_i(t):= \frac{\sigma}{N(x_{i+1}(t) - x_i(t))}, \qquad i=0,\ldots,N-1.
\end{equation*}
For future use we compute 
\begin{equation}\label{Ridot}
\dot{R}_i(t) = - N R_i^2(t)(\dot{x}_{i+1}(t) - \dot{x}_i(t)).
\end{equation}
According to our formal considerations in Section 1, the functions $R_i$ are good candidates to be considered as approximation for solution of \eqref{main} in the \emph{space} interval $(x_{i+1}(t), x_i(t))$. Hence, we define the $N$-discrete density as 
\begin{equation}\label{rhoN}
\rho^N(t,\,x):= \sum_{i=0}^{N-1}  R_i (t) \chi_{[x^N_i(t),\,x^N_{i+1}(t))}(x), \quad \mbox{ for $(t,x) \in [0,T] \times [0,\ell]$.}
\end{equation}  
Note that, by construction, $\rho^N$ has the same mass of $\bar{\rho}$.

In the following Definition we introduce the notion of weak solutions we are dealing with.

\begin{defn}[Weak Solution]\label{weaksolutiondfn}
A function $\rho \in L^\infty([0,T]\times[0,\ell])$ is a \emph{weak solution} of the Initial-boundary value Problem \eqref{main} if  $\rho(0,\cdot)= \bar{\rho}$ and $\rho$ satisfies
\[ \int_0^T \int_0^{\ell} [\rho \varphi_t + \phi(\rho)\varphi_{xx} - \rho v(\rho) K'\ast \rho\, \varphi_x] dxdt = 0, \]
for all $\varphi \in C^{\infty}_c([0,T]\times[0,\ell])$ such that $\varphi_x(\cdot,0) = \varphi_x(\cdot,\ell) = 0$.
\end{defn}

The main result of our work is stated in the following Theorem.

\begin{thm}\label{maintheorem}
Consider $\phi,\,v,\,K$ satisfying the conditions (Diff), (Vel) and (Ker) respectively and let $\bar{\rho}: [0,\ell] \to \R$ be as in (In1) and $(In2)$ and let $T>0$ be fixed. Then the discretized densities $\rho^N$ as defined in \eqref{rhoN} strongly converge up to a subsequence in $L^1([0,T]\times[0,\ell])$ to a limit $\rho \in L^{\infty} \cap BV ([0,T]\times[0,\ell])$ which solves the Initial-boundary value Problem \eqref{main} in the sense of Definition \ref{weaksolutiondfn}.
\end{thm}

Let us introduce now the main concepts about one the dimensional Wasserstein distances, see \cite{AGS,S,V2} for further details.
As already mentioned, we deal with probability densities with constant mass in time and we need to evaluate their distances at different times in the Wasserstein sense.

For a fixed mass $\sigma>0$, we consider the space
\begin{equation*}
  \mathfrak{M}_\sigma \doteq \bigl\{\mu \hbox{ Radon measure on $\R$ with compact support} \colon \mu\ge 0 \text{ and }\mu(\R)=\sigma \bigr\}.
\end{equation*}
Given $\mu\in \mathfrak{M}_\sigma$, we introduce the pseudo-inverse variable $X_\mu \in L^1([0,\sigma];\R)$ as
\begin{equation}\label{eq:pseudoinverse}
X_\mu(z) \doteq \inf \bigl\{ x \in \R \colon \mu((-\infty,x]) > z \bigr\}.
\end{equation}

For $\sigma=1$, the one-dimensional \emph{$1$-Wasserstein distance} between $\rho_1,\rho_2\in \mathfrak{M}_1$  see e.g.\ \cite{CT} can be defined as
\[
d_{W^1}(\rho_1,\rho_2) \doteq \|X_{\rho_1}-X_{\rho_2}\|_{L^1([0,1];\R)}.
\]
We introduce the \emph{scaled $1$-Wasserstein distance} between $\rho_1,\rho_2\in \mathfrak{M}_\sigma$ as
\begin{equation}\label{eq:wass_equiv0}
    d_{W^1,\sigma}(\rho_1,\rho_2)\doteq \|X_{\rho_1}-X_{\rho_2}\|_{L^1([0,\sigma];\R)} .
\end{equation}
A sequence $(\rho_n)_n$ in $\mathfrak{M}_{\sigma}$ converges to $\rho\in \mathfrak{M}_{\sigma}$ in $d_{W^1,\sigma}$ if and only if for any $\varphi\in C_0([0,\ell];\R)$ growing at most linearly at infinity
\[
\lim_{n\to+\infty}\int_0^\ell \varphi(x) \,d\rho_n(x) = \int_0^\ell \varphi(x) \,d\rho(x) .
\]

We now state a technical result which will serve in the sequel. 
%qui scrivere che la versione è adattata al nostro paper
\begin{thm}[Generalized Aubin-Lions Lemma,\cite{RS}]\label{aubinlions}
Let $\tau>0$ be fixed. Let $\eta^N$ be a sequence in $L^{\infty}((0,\,\tau); L^1(\R))$ such that 
$\eta^N(t,\,\cdot) \geq 0$ and $\| \eta^N(t,\,\cdot) \|_{L^1(\R)}=1$ for every $N\in\mathbb{N}$ and $t\in [0,\,\tau]$.
If the following conditions hold
\begin{enumerate}
\item[I)] $\sup_{N}  \int_0^{\tau} \left[\|\eta^N(t,\,\cdot)\|_{L^1(\R)}dt + TV\big[ \eta^N(t,\,\cdot)\big]+ \mathrm{meas}(\mathrm{supp}[\eta^N(t,\cdot)])\right]dt < \infty$,
\item[II)] there exists a constant $C>0$ independent from $N$ such that $d_{W^1}\big( \eta^N(t,\,\cdot), \eta^N(s,\,\cdot) \big) < C |t-s|$ for all $s,\,t \in (0,\,\tau)$,
\end{enumerate}
then $\eta^N$ is strongly relatively compact in $L^1([0,\,\tau]\times \R)$.
\end{thm}

We conclude this Section showing that the unique solution to \eqref{ODES} is well defined for every $t \in [0,T]$. It is enough to prove that the distances $x_{i+1}(t)-x_i(t)$ never degenerate. However, in our case we can prove something stronger, i.e. the discrete densities never exceed a uniform bound from above and below 
that depends only on $\bar{\rho}$ and $T$.

\begin{prop}[Discrete Min-Max Principle]\label{st:MinMax}
Let $T>0$ be fixed and $\bar{\rho}$ under the assumptions (In1) and (In2). Let $c$ be a positive consant satisfying $c > \frac{2m\,v_{max}\,L\ell}{\sigma}$. Then 
\begin{equation}\label{minmax}
\frac{\sigma}{MN} \leq x_{i+1}(t) - x_i(t) \leq 2\frac{e^{cT}\sigma}{mN},
\end{equation} 
for every $i=0,\ldots,N-1$ and $t \in [0,T]$.
\end{prop}
\begin{proof}
Notice that, since $m \leq \bar{\rho} \leq M$, then by definition $\frac{\sigma}{MN} < x_{i+1}(0) - x_i(0) < \frac{\sigma}{mN}$. We want to show that similar bounds are preserved during the evolution. Let us first focus on the lower one. Thanks to the regularity of $K',v$ and $\phi$ and since $x_{i+1}(0) - x_i(0) \geq \frac{\sigma}{MN}$ for $i=0,\ldots,N-1$, we deduce that $\dot{x}_i$ is continuous at least on a short time interval. Let now $t_1$ be the first instant for which there is an index $i$ such taht $x_{i+1}(t_1) - x_i(t_1) = \frac{\sigma}{MN}$ and assume, for sake of contradiction, that there exists $t_2 > t_1$ such that $x_{i+1}(t) - x_i(t) < \frac{\sigma}{MN}$ for $t \in (t_1,t_2]$.
Since $x_{j+1}(t_1) - x_j(t_1) \geq \frac{\sigma}{MN}$ for all $j \neq i$, we obtain that $R_i(t_1) \geq R_j(t_1)$ for every $j \neq i$ and, in particular, 
\[ R_i(t_1) \geq R_{i-1}(t_1)\, \mbox{ and }\, R_i(t_1) \geq R_{i+1}(t_1)\,.  \]
Now, from the monotonicity of $\phi$, the sign of $K'$ and (Vel), we deduce that $v(R_i(t_1))=0$ and
\begin{align*}
\dot{x}_{i+1}(t_1) &= N(\phi(R_i(t_1)) - \phi(R_{i+1}(t_1))) - \frac{v(R_{i+1}(t_1))}{N}\sum_{j>i+1} K'(x_{i+1}(t_1)-x_j(t_1)) \geq 0,\\
\dot{x}_i(t_1) &= N(\phi(R_{i-1}(t_1))- \phi(R_i(t_1))) - \frac{v(R_{i-1}(t_1))}{N} \sum_{j<i} K'(x_i(t_1) - x_j(t_1)) \leq 0.
\end{align*}
Then the regularity of $\dot{x}_i$ and $\dot{x}_{i+1}$ ensures that 
\[ x_{i+1}(t) - x_i(t) \geq x_{i+1}(t_1) - x_i(t_1) = \frac{\sigma}{MN}, \]
for all $t \in (t_1,t_1+ \varepsilon]$ for some $\varepsilon>0$ small enough and this contradicts the existence of $t_2$. A consequence of the lower bound of~\eqref{minmax} and of the zero velocity boundary condition  is that  all the particles stay inside the domain $[0,\ell]$ and maintain their initial order for all times. 

We focus now on the upper bound in~\eqref{minmax}. Let us call
\[ \tau_1:= \inf \{ s \in (0,T] :\,\exists\,i\,:\,x_{i+1}(s) - x_i(s) > e^{cs}\frac{\sigma}{mN} \}, \]
If $\tau_1=T$ then the right inequality of \eqref{minmax} follows trivially
because $e^{c\tau_1} = e^{cT}$. It remains to discuss the case when $\tau_1 < T$. For sake of contradiction, assume that there is $\tau_2 \in (\tau_1,T]$ such that  
\begin{equation}\label{assurdoprinmin}
x_{i+1}(t) - x_i(t) > e^{ct} \frac{\sigma}{mN} \quad \mbox{ for every $t \in (\tau_1,\tau_2]$.}
\end{equation}
The contradiction occurs as soon as we prove that 
\begin{equation}\label{ddtesp}
\frac{d}{dt}\left[ e^{-ct}(x_{i+1}(t) - x_i(t))\right]_{|_{t=\tau_1}} < 0\,.
\end{equation}
Indeed, being this function smooth, there exists some positive $\varepsilon \ll 1$ for which $\tau_1 + \varepsilon < \tau_2$ and $\frac{d}{dt}\left[ e^{-ct}(x_{i+1}(t) - x_i(t))\right] < 0$ for all $t \in (\tau_1,\tau_1+\varepsilon]$. Then, for such $t$,~\eqref{ddtesp} implies
\[ e^{-ct}(x_{i+1}(t) - x_i(t)) = \frac{\sigma}{mN} + \int_{\tau_1}^t \frac{d}{ds}\left[ e^{-cs}(x_{i+1}(s) - x_i(s))\right]ds \leq \frac{\sigma}{mN}\,,  \]
which clearly contradicts~\eqref{assurdoprinmin}. Therefore, the above estimate of~\eqref{minmax} follows if we show the validity of~\eqref{ddtesp}. At the time $\tau_1$ we know that 
\[ x_{i+1}(\tau_1) - x_i(\tau_1) = e^{c\tau_1}\frac{\sigma}{mN}\quad \mbox{ and }\quad x_{j+1}(\tau_1) - x_j(\tau_1) \leq e^{c\tau_1}\frac{\sigma}{mN}\quad \forall\,j\neq i,\]
thus, in particular,
\[ R_{i+1}(\tau_1) \geq R_i(\tau_1)\quad \mbox{ and }\quad R_{i-1}(\tau_1) \geq R_i(\tau_1)\,. \]
Then the assumptions on (Diff),(Vel) and (Ker) ensure that
\begin{align*}
\frac{d}{dt}&\left[ e^{-ct}(x_{i+1}(t) - x_i(t))\right]_{|_{t=\tau_1}} = e^{-c\tau_1}\left[(\dot{x}_{i+1}(\tau_1) -\dot{x}_i(\tau_1)) - c(x_{i+1}(\tau_1) - x_i(\tau_1))\right] \\
&= e^{-c\tau_1} N[(\phi(R_i(\tau_1)) - \phi(R_{i+1}(\tau_1)))-(\phi(R_{i-1}(\tau_1)) - \phi(R_i(\tau_1)))] - c\frac{\sigma}{mN} \\
&\, -e^{-c\tau_1}\left[\frac{v(R_{i+1}(\tau_1))-v(R_i(\tau_1))}{N} \sum_{j>i+1} K'(x_{i+1}-x_j) -\frac{v(R_{i-1}(\tau_1))-v(R_i(\tau_1))}{N} \sum_{j<i} K'(x_i-x_j)\right] \\
&\, -e^{-c\tau_1}\frac{v(R_i(\tau_1))}{N} \sum_{j \neq i,\,i+1} [K'(x_{i+1}-x_j) - K'(x_i-x_j)] -2e^{-c\tau_1}\frac{v(R_i(\tau_1))}{N} K'(x_{i+1}(\tau_1) -x_i(\tau_1)) \\
&\leq e^{-c\tau_1} \frac{2v_{max}L \ell}{N} - c\frac{\sigma}{mN} 
%= \frac{e^{-c\tau_1}}{N}\left(2v_{max}L \ell - c \frac{\sigma e^{c\tau_1}}{m}\right) 
\leq \frac{e^{-c\tau_1}}{N}(2v_{max}L \ell - c\frac{\sigma}{m}) < 0\,, 
\end{align*}
where the last inequality holds because of our initial choice of $c$. 
\end{proof}

\begin{rem}
Proposition \ref{st:MinMax} easily implies a uniform bound in time on $\phi$. Indeed, 
thanks to~\eqref{minmax} we obtain 
\[
e^{-cT}m \leq R_i(t) \leq M \mbox{ for all $i=0,\ldots,N-1$, } 
\]
and from the regularity of $\phi$ required in (Diff) we deduce the existence of a constant $C= C(M,T) > 0$ such that
\begin{equation}\label{boundphi}
\sup_{t\in[0,T]} \phi(R_i(t)) < C, 
\end{equation}
for every $i=0,\ldots,N-1$.
\end{rem}

\section{Proof of the Main Result}

The proof of Theorem~\ref{maintheorem} relies on two main steps: the first is prove that $(\rho^N)$ is strongly converging to a limit $\rho$ in  $L^1([0,T] \times [0,\ell])$, the second is show that $\rho$ is a weak solution of the initial-boundary value problem~\eqref{main}. In this section we take care of both these steps. As we will show in Proposition~\ref{totalvariation}, the sequence $(\rho_N)$ satisfies good compactness estimates with respect to the space variables. On the other hand, 
we do not have good control on the time oscillations of the discrete densities, thus the $L^1$-compactness in the product space is not straightforward. Nevertheless, Theorem~\ref{aubinlions} ensures that a uniform continuity estimate in time of the 
$1$-Wasserstein distance is enough to pass to the limit. This will be the content of Proposition~\ref{continuitytime}.
Summarizing, the main result of the first part of the section is the following.

\begin{thm}\label{convergence}
Under the assumptions of Theorem~\ref{maintheorem}, there exists a probability density $\rho: [0,T] \times [0,\ell] \to [m,M]$ such that $\rho^N$ $L^1$-converges strongly to $\rho$ in the product topology.
\end{thm}

In the second part of the section, instead, we take care of the second step. Indeed, we show that the limit $\rho$ obtained in Theorem~\ref{convergence} is solves Problem~\eqref{main} in the sense of Definition~\ref{weaksolutiondfn}. 

In the following two Propositions we prove a uniform bound on the total variations of the discrete densities $\rho^N$ and a uniform Lipschitz control 
of the $1$-Wasserstein distance with respect to the time variable. For simiplicity and w.l.g., we assume from now on that $\sigma=1$.

\begin{prop}\label{totalvariation}
Let $T>0$ and $\bar{\rho},\,v,\phi,\,K$ under the assumptions of Theorem~\ref{maintheorem}. Then for every $N \in \mathbb{N}$ one has 
\begin{equation}\label{TV}
TV[\rho^N(t,\,\cdot)] \leq TV[\bar{\rho}] C(K,T) \qquad \forall t \in [0,T]\,,
\end{equation}
where $C(K,T)$ is a positive constant depending only on the Lipschitz property of $K'$ and on the final evolution time $T$.
\end{prop}
\begin{proof}
It is easy to show that $$TV[\rho^N(0,\,\cdot)] \leq TV[\bar{\rho}].$$ 
The total variation of $\rho^N$ at time $t$ is exactly 
\begin{align*}
T&V[\rho^N(t,\,\cdot)] = R_0(t) + R_{N}(t) + \sum_{i=0}^{N-1} |R_{i+1}(t) - R_i(t)| \\
&=\sum_{i=1}^{N-1}R_i [sign(R_i - R_{i-1}) - sign(R_{i+1}-R_i)] - R_0 (sign(R_1-R_0) -1) + R_N( sign(R_N - R_{N-1}) + 1) \\
&=\big( 1+ sign(R_0 -R_1) \big)R_0 + \big( 1- sign(R_{N-1} -R_N) \big)R_N +  \sum_{i=1}^{N-1}R_i \mu(R_i),
\end{align*}
where we set for brevity $\mu(R_{0})=(1-sign(R_{1}(t) - R_{0}(t)))$, $\mu(R_{N})=(1+sign(R_{N}(t) - R_{N-1}(t))) $ and
\[   \mu(R_i(t)):=sign(R_i(t) - R_{i+1}(t)) - sign(R_{i-1}(t) - R_i(t))\,.\]
The idea is to obtain~\eqref{TV} with a Gronwall type argument, therefore we compute
\begin{align*}
 \frac{d}{dt} TV[\rho^N(t,\,\cdot)] &= \dot{R}_0(t) +\dot{R}_{N}(t) + \sum_{i=0}^{N-1} sign\big(R_{i+1}(t) - R_i(t)\big)\big( \dot{R}_{i+1}(t) - \dot{R}_i(t) \big) \\
&= \big( 1+ sign(R_0(t) -R_1(t)) \big)\dot{R}_0(t) + \big(1 - sign(R_{N-1} - R_N(t))\big) \dot{R}_N(t) \\
&\,+ \sum_{i=1}^{N-1} \big(sign(R_i(t) - R_{i+1}(t)) - sign(R_{i-1}(t) - R_i(t)) \big)\dot{R}_i(t)\,.
\end{align*}
The value of the coefficient $\mu(R_i(t))$ clearly depends on the positions of the consecutive particles, it is easy to see that for $i \in \{ 1,\,\ldots,\,N-1\}$
\begin{equation*}
\mu(R_i(t))= \left\lbrace\begin{array}{lll}
-2 \quad &\mbox{if $R_{i+1} > R_i$ and $R_{i-1}> R_i$},\\
2 \quad  &\mbox{if $R_{i+1} < R_i$ and $R_{i-1}< R_i$},\\
0 \quad  &\mbox{if $R_{i+1} > R_i > R_{i-1}$ or $R_{i-1}> R_i > R_{i-1}$,}
\end{array}
\right.
\end{equation*}
therefore
\begin{equation*}
1+sign(R_0-R_1)=\left\lbrace \begin{array}{ll}
0 \quad \mbox{if $R_1 < R_0$,}\\
2  \quad \mbox{if $R_1 > R_0$,}
\end{array} \right.
\qquad
1- sign(R_{N-1} -R_N)= \left\lbrace\begin{array}{ll}
0 \quad \mbox{if $R_{N-1} > R_N$,}\\
2  \quad \mbox{if $R_{N-1} < R_N$.}
\end{array}\right. 
\end{equation*}
Recalling the explicit formula of $\dot{R}_i$ computed before, we can rewrite 
\[  \sum_{i=1}^{N-1}\mu(R_i(t)) \dot{R}_i(t) = - \sum_{i=1}^{N-1}\mu(R_i(t))(R_i(t))^2 \emph{I}_i -   \sum_{i=1}^{N-1}\mu(R_i(t))R_i(t)\emph{II}_i - \sum_{i=1}^{N-1}\mu(R_i(t))(R_i(t))^2 \emph{III}_i\,,  \]
where 
\[ \emph{I}_i = -\big(v(R_{i+1}(t)) - v(R_i(t))\big)\sum_{j > i+1} K'(x_{i+1}(t) - x_j(t)) - \big(v(R_i(t)) - v(R_{i-1}(t))\big)\sum_{j < i} K'(x_i(t) - x_j(t))\,, \]
and
\[ \emph{II}_i = -R_i(t)v(R_i(t)) \sum_{j \neq i,\,i+1} \big( K'(x_{i+1}(t) - x_j(t)) - K'(x_i(t) - x_j(t)) \big) - 2R_i(t)v(R_i(t))K'(x_{i+1}(t)-x_i(t))\,, \]
and 
\[ \emph{III}_i = N^2[(\phi(R_i(t))- \phi(R_{i+1}(t)))-(\phi(R_{i-1}(t))-\phi(R_i(t)))].\]
Then the time derivative of the total variation can be reads as 
\[  \frac{d}{dt} TV[\rho^N(t,\,\cdot)] = \mu(R_0)\dot{R}_0(t) + \mu(R_N) \dot{R}_N(t) - \sum_{i=1}^{N-1} \mu(R_i(t))(R_i(t))^2 (\emph{I}_i + \emph{III}_i) -  \sum_{i=1}^{N-1} \mu(R_i(t))R_i(t) \emph{II}_i\,, \]
Then estimate~\eqref{TV} follows thanks to a Gronwall argument as soon as we show that
\begin{equation}\label{dTVlimitata}
 \frac{d}{dt} TV[\rho^N(t,\,\cdot)] \leq C_1 + C_2\,TV[\rho^N(t,\,\cdot)]\,,
\end{equation}
for some constants $C_{1/2}=C_{1/2}(K,L,\ell)$.
We will see that the first three terms will not cause problems since $\dot{R}_0,\,\dot{R}_N$ are bounded and $-\sum_{i=1}^{N-1}\mu(R_i(t))(R_i(t))^2(\emph{I}_i + \emph{III}_i)$ is always negative, while it will be less easy to show that the last term satisfies the desired Gronwall type estimate. \\
Let us start with $-\sum_{i=1}^{N-1}\mu(R_i(t))(R_i(t))^2(\emph{I}_i + \emph{III}_i)$. We can already observe that the only relevant contributions in the sum come from the particles $x_i$ for which $\mu(R_i(t))\neq 0$. However, if the index $i$ is such that $\mu(R_i(t))=-2$, then $R_{i+1}, R_{i-1} > R_i$ and the monotonicity of $v$ and $\phi$ imply 
\begin{align*} 
 & v(R_{i+1}(t)) - v(R_i(t))  <0,\quad v(R_i(t))-v(R_{i-1}(t)) >0, \\ 
 & \phi(R_i(t)) - \phi(R_{i+1}(t)) <0,\quad \phi(R_{i-1}(t)) - \phi(R_i(t)) >0 .
 \end{align*}
The sign of $K'$ ensures that $\emph{I}_i < 0$, thus, on the other hand, $-2(R_i (t))^2 \emph{I}_i  >0$.
An analogous argument implies that, if $i$ is such that $\mu(R_i(t))=2$, then $\emph{I}_i > 0$ and $2(R_i (t))^2 \emph{I}_i  >0$. These considerations lead immediately to 
\begin{equation}\label{stuck[I]}
-\sum_{i=1}^{N-1}\mu(R_i(t))(R_i(t))^2(\emph{I}_i + \emph{III}_i)< 0\,.
\end{equation}
Let us now focus on $-\sum_{i=1}^{N-1}\mu(R_i(t))R_i(t)\emph{II}_i$. In this case, we would like to obtain an upper bound in terms of $TV[\rho^N(t,\,\cdot)]$ and for this purpose we need to estimate $| II_i |$. We get
\begin{align}\label{stimsuII_i}
\notag |\emph{II}_i| &=   R_i(t)|v(R_i(t))| \left|-2K'(x_{i+1}(t) - x_i(t)) -\sum_{j \neq i,\,i+1} \big(K'(x_{i+1}(t) - x_j(t)) - K'(x_i(t) - x_j(t))  \big)\right| \\
&\leq R_i(t) LC \frac{N-2}{N} \frac{1}{R_i(t)} + 2L\frac{1}{N} \leq C\,,
\end{align}
for some constant $C>0$. 
% %%%%%%%%%%%%%
% %% pezzo nuovo by Manu
% %%%%%%%%%%%%%%%%%%%%%%%
% \textcolor{red}{NEED TO ASSUME THAT K''' IS BOUNDED} 
% %%%%%%%%%%%%%%%%%%%%%%
We have
\begin{equation*}
-\sum_{i=1}^{N-1}\mu(R_i(t))R_i(t)\emph{II}_i = B.T. + \sum_{i=1}^{N-2} sign(R_{i-1}-R_i)(R_i-R_{i-1})II_i + \sum_{i=2}^{N-2} sign(R_{i-1}-R_i) R_i (II_i - II_{i-1}),
\end{equation*}
and thanks to~\eqref{stimsuII_i} it is easy to see that 
\begin{equation}\label{stuck[IIa]}
|B.T.| + \left| \sum_{i=1}^{N-2} sign(R_{i-1}-R_i)(R_i-R_{i-1})II_i\right| \leq C_1 + C_2\sum_{i=2}^{N-1}|R_i - R_{i-1}| \leq C_1 + C_2 TV[\rho^N(t)],
\end{equation}
then it remains to check the term involving $II_i - II_{i-1}$. 
It is easy to see that $II_i - II_{i-1}$ may be written as a sum of three terms $II_{A_i} + II_{B_i} + II_{C_i}$, where 
\begin{align*}
II_{A_i} &= (R_{i-1}v(R_{i-1})- R_iv(R_i))\left[2K'(x_{i+1}-x_i) + \sum_{j\neq i,i+1}(K'(x_{i+1}-x_j) - K'(x_i-x_j)) \right], \\
II_{B_i} &= 2R_{i-1}v(R_{i-1})(K'(x_i-x_{i-1})-K'(x_{i+1}-x_i)), \\
II_{C_i} &= R_{i-1}v(R_{i-1})\left[\sum_{j\neq i-1,i}(K'(x_i-x_j)-K'(x_{i-1}-x_j)) -\sum_{j\neq i+1,i}(K'(x_{i+1}-x_j)-K'(x_{i}-x_j))\right]. 
\end{align*}
We can notice immediately that 
\begin{equation}\label{Bi}
|II_{B_i}| \leq 2L\|v\|_{L^{\infty}} R_{i-1}|(x_{i-1}-x_i)- (x_{i+1}-x_i)| = \frac{2L}{N}\|v\|_{L^{\infty}} R_{i-1} \frac{|R_{i-1}-R_i|}{R_{i-1}R_i}, 
\end{equation}
while, recalling that the functions $f(z)=zv(z)$ and $K'$ are Lipschitz, 
\begin{equation}\label{Ai}
|II_{A_i}| \leq Lip[f] NL(x_{i+1}-x_i) |R_i - R_{i-1}|.
\end{equation}
On the other hand, 
\begin{equation*}
|II_{C_i}| \leq |\widetilde{BT_i}| + \|v\|_{L^{\infty}} R_{i-1}\sum_{j\neq i \pm 1,i}|2K'(x_i-x_j) - K'(x_{i+1}-x_j) - K'(x_{i-1}-x_j)|,
\end{equation*}
so, if we expand $K'(x_{i \pm 1}- x_j)$ at the first order w.r.t. $K'(x_i - x_j)$ and recall that $K''$ and $K'''$ are bounded in $[0,\ell]$, we get 
\begin{align}\label{Ci}
\notag R_i|II_{C_i}| &\leq |R_i \widetilde{BT_i}| + \|v\|_{L^{\infty}} R_iR_{i-1}\sum_{j\neq i \pm 1,i} |K''(x_i-x_j)||(x_i-x_{i-1})-(x_{i+1}-x_i)| \\
\notag &\quad + \frac{\|v\|_{L^{\infty}}}{2} R_i R_{i-1} \sum_{j\neq i \pm 1,i} \|K'''\|_{L^{\infty}([-\ell,\ell])} [(x_{i-1}-x_i)^2 + (x_{i+1}-x_i)^2] \\
&\leq \tilde{C} + \frac{\|v\|_{L^{\infty}}}{N} L|R_i-R_{i-1}| + \frac{\|v\|_{L^{\infty}}}{2} L (R_i(x_i-x_i-1) + R_{i-1}(x_{i+1}-x_i)).
\end{align}
Thanks to~\eqref{Bi},~\eqref{Ai} and~\eqref{Ci} and the fact that the support of $\rho^N$ is uniformly bounded in time for every $N$, we then obtain
\begin{align*}
\sum_{i=2}^{N-2} &|sign(R_{i-1}-R_i) R_i (II_i - II_{i-1})| \leq \sum_{i=2}^{N-2} R_i[|II_{A_i}|+|II_{B_i}|+|II_{C_i}|] \\
&\leq C +  CL (Lip[f] + \frac{2\|v\|_{L^{\infty}}}{N} + \|v\|_{L^{\infty}})\sum_{i=2}^{N-2}|R_i - R_{i-1}| + L\|v\|_{L^{\infty}}\sum_{i=0}^{N-1}|x_{i+1}-x_i| \\
&\leq C(1 + TV[\rho^N(t,\cdot)]),
\end{align*}
and, together with~\eqref{stuck[IIa]}, this implies
\begin{equation}\label{stuck[II]}
\left|  -\sum_{i=1}^{N-1}\mu(R_i(t))R_i(t)\emph{II}_i  \right| \leq C_1 + C_2\,TV[\rho^N(t,\,\cdot)].
\end{equation}
%%% sono arrivata qui
%%% sotto è da sistemare
Consider then with $\dot{R}_0$ and $\dot{R}_N$. The argument is similar, then we deal only with $\mu(R_0)\dot{R}_0$. Since $(v(R_1) - v(R_0)) \leq 0$, we can compute
\begin{align*}
\mu(R_0)\dot{R}_0 &= \mu(R_0)R_0 [R_0v(R_1)\sum_{j>1}\big(K'(x_1-x_j) - K'(x_0-x_j)\big) +2R_0v(R_0)K'(x_1-x_0)] \\
&\quad + \mu(R_0)(R_0)^2(v(R_1) - v(R_0))\sum_{j>1} K'(x_0 - x_j) \\
&\leq  \mu(R_0)R_0 [R_0v(R_1)\sum_{j>1}\big(K'(x_1-x_j) - K'(x_0-x_j)\big) +2R_0v(R_0)K'(x_1-x_0)]\,;
\end{align*}
and also
\[ \left| R_0v(R_1)\sum_{j>1}\big(K'(x_1-x_j) - K'(x_0-x_j)\big) +2R_0v(R_0)K'(x_1-x_0)\right| \leq CL\frac{N-1}{N} + \frac{2LC}{N}\,.  \]
In particular, $\mu(R_0)\dot{R}_0 \leq (3CL) R_0$ and
\begin{equation}\label{primoeultimotermine}
\mu(R_0)\dot{R}_0 + \mu(R_N)\dot{R}_N \leq 3CL (R_0 + R_N) \leq (3CL) TV[\rho^N(t,\,\cdot)]\,.
\end{equation}
By putting together~\eqref{stuck[I]},~\eqref{stuck[II]} and~\eqref{primoeultimotermine} we get estimate~\eqref{dTVlimitata} and~\eqref{TV} follows as a consequence of Gronwall Lemma. 
\end{proof}

\begin{prop}\label{continuitytime}
Let $T>0$ and $\bar{\rho},\,v,\phi,\,K$ under the assumptions of Theorem~\ref{maintheorem}. Then there exists $C>0$ that does not depend on $N$ such that 
\begin{equation}\label{contime}
d_{W^1}\big( \rho^N(t,\,\cdot), \rho^N(s,\,\cdot) \big) < C |t-s| \quad \forall\, s,\,t \in (0,\,T),\, \forall\,N \in \mathbb{N}\,.
\end{equation}
\end{prop}

\begin{proof}
Assume without loss of generality that $0< s < t < T$. 
We want to investigate the continuity in time of the discrete density $\rho^N$ with respect to the $1$-Wasserstein distance. Despite this step is more involved in higher dimensions, in the one-dimensional case we can take advantage of the well known relation between the $1$-Wasserstein distance of two probability measures and the $L^1$ distance of their respective pseudo-inverse functions, see Section 2. The claim follows once we show that
\[  \|  X_{\rho^N(t,\cdot)} - X_{\rho^N(s,\,\cdot)} \|_{L^1([0,1])} < C|t-s|, \]
for all $s,\,t \in (0,T)$ independently on $N$, where $X_{\rho^N(t,\cdot)}$ is the pseudo-inverse of the discretize density.
By the definition of $\rho^N$ we can explicitely compute
\[ X_{\rho^N(t,\,\cdot)}(z) = \sum_{i=0}^{N-1} \left(x_i(t) + \left(z-i\frac{1}{N}\right) \frac{1}{R_i(t)}\right) \textbf{1}_{[i\frac{1}{N},\,(i+1)\frac{1}{N})}(z)\,. \]  
Let us observe that 
\[ |\dot{x}_i(t)| \leq  N Lip[\phi]|R_i(t) - R_{i-1}(t)| + \left|\frac{v(R_i(t))}{N} \sum_{j>i} K'(x_i - x_j) - \frac{v(R_{i-1})}{N}\sum_{j<i} K'(x_i - x_j) \right|, \]
and, thanks to~\eqref{TV}, we obtain
\begin{equation}\label{stimaxdot}
\notag\sum_{i=0}^N |\dot{x}_i(t)| \leq 2N^2 C(K,\phi,T,\bar{\rho}).
\end{equation}  
Having in mind that \eqref{stimaxdot} and
\[ \left| \frac{d}{d\tau} \frac{1}{R_i(\tau)}  \right| = N |\dot{x}_{i+1}(\tau) - \dot{x}_i(\tau)| \leq N |\dot{x}_{i+1}(\tau)| + N |\dot{x}_i(\tau)|\,, \]
we directly conclude 
\begin{align*}
d_{W^1}\big( \rho^N(t,\,\cdot), \rho^N(s,\,\cdot) \big) &= \| X_{\rho^N(t,\,\cdot)} - X_{\rho^N(s,\,\cdot)}  \|_{L^1([0,\,1])} \\
&\leq \sum_{i=0}^{N-1} \int_{i/N}^{(i+1)/N} \left| x_i(t) - x_i(s) + \left(z - \frac{i}{N} \right) \left(\frac{1}{R_i(t)} - \frac{1}{R_i(s)} \right)  \right| dz \\
&\leq \sum_{i=0}^{N-1} \frac{1}{N} |x_i(t) - x_i(s)| +  \sum_{i=0}^{N-1} \left|\frac{1}{R_i(t)} - \frac{1}{R_i(s)} \right|  \int_{i/N}^{(i+1)/N} \left(z - \frac{i}{N} \right)dz \\
&= \sum_{i=0}^{N-1} \frac{1}{N} |x_i(t) - x_i(s)|  + \sum_{i=0}^{N-1} \frac{1}{2N^2} \int_s^t \left| \frac{d}{d\tau} \frac{1}{R_i(\tau)}\right| d\tau \\
&\leq  \frac{3}{N}\sum_{i=0}^{N} \int_s^t \left|\dot{x}_i(\tau)\right| d\tau \leq C(K,\phi,T,\bar{\rho})|t-s|.
\end{align*}
\end{proof}

\begin{proof}[Proof of Theorem~\ref{convergence}]
By construction, the densities $\rho^N$ are always nonnegative, with constant mass $\| \rho^N(t,\cdot)\|_{L^1([0,\ell])}=1$ and they are bounded uniformly in $L^{\infty}([0,T]\times[0,\ell])$ thank to~\eqref{minmax}.
Moreover, estimates~\eqref{TV} and~\eqref{contime} ensure that conditions I and II in Theorem~\ref{aubinlions} hold and $\rho^N$ converges, up to a subsequence, to a strong $L^1$-limit $\rho$ in the product space $[0,T]\times[0,\ell]$. Finally, from~\eqref{minmax} we deduce that $m \leq \rho(t,x) \leq M$ everywhere in $[0,T]\times[0,\ell]$, then the proof is concluded.
\end{proof}

We are finally in position to prove Theorem \ref{maintheorem}. We deduce from Theorem~\ref{convergence} the existence of a strong $L^1$-limit of the approximations, then we only need to show that this limit is a weak solution of \eqref{main} in terms of Definition~\ref{weaksolutiondfn}.

\begin{proof}[Proof of Theorem~\ref{maintheorem}]
Thanks to Theorem~\ref{aubinlions} we know that, up to a subsequence, $\rho^N$ converges to $\rho$ in $L^1([0,\,T]\times [0,\ell])$. There would be nothing to prove if we knew that for every $N$ the discrete density $\rho^N$ is a weak solution of~\eqref{main}. Unfortunately, this is not true in general but we can prove that the mismatch is vanishing as $N \to \infty$. We claim that for every $\varphi \in C_c^\infty( [0,\,T] \times [0,\ell])$  such that $\varphi_x(\cdot,0) = \varphi_x(\cdot,\ell) = 0$ one has
\begin{equation}\label{weaksol}
\int_0^T \int_0^\ell \rho^N(t,\,x) \varphi_t (t,\,x) + \phi(\rho^N)\varphi_{xx} - \rho^N(t,\,x) v(\rho^N(t,\,x))K' \ast \rho^N(t,\,x) \varphi_x(t,\,x) dx dt \longrightarrow  0,
\end{equation}
as $N\to\infty$.
In order to simplify the notation, we omit the dependence on the variable $t$ whenever it is clear from the context. 
Substituting the definition of $\rho^N$, the l.h.s. of \eqref{weaksol} becomes 
\begin{align*}
&\sum_{i=0}^{N-1}\int_0^T \int_{x_i}^{x_{i+1}} R_i\varphi_t(x)dxdt + \int_0^T \sum_{i=0}^{N-1}\phi(R_i)(\varphi_x(x_{i+1})-\varphi_x(x_i))dt \\
&- \int_0^T \sum_{i=0}^{N-1}\int_{x_i}^{x_{i+1}} R_i v(R_i)\varphi_x(x)\left[ \sum_{j=0}^{N-1} R_j (K(x-x_{j+1})-K(x-x_j)) \right]dxdt.
\end{align*}   
Let us observe that, integrating by part twice and using~\eqref{Ridot}, the first term in the above expression can be rewritten as follows
\begin{align*}
 \int_0^T\sum_{i=0}^{N-1} \int_{x_i}^{x_{i+1}} R_i\varphi_t(x) dxdt =& 
% &= \int_0^T \sum_{i=0}^{N-1} R_i \frac{d}{dt} \left( \int_{x_i}^{x_{i+1}} \varphi(t,\,x) dx\right) dt -  \int_0^T \sum_{i=0}^{N-1} R_i \left( \varphi(t,\,x_{i+1})\dot{x}_{i+1} - \varphi(t,\,x_i)\dot{x}_i \right)dt \\
% &= -\int_0^T \sum_{i=0}^{N-1} \dot{R}_i \int_{x_i}^{x_{i+1}}\varphi(t,\,x) dxdt -  \int_0^T \sum_{i=0}^{N-1} R_i \left( \varphi(t,\,x_{i+1})\dot{x}_{i+1} - \varphi(t,\,x_i)\dot{x}_i \right)dt \\
 \int_0^T \sum_{i=0}^{N-1} R_i (\dot{x}_{i+1} - \dot{x}_i) \intmed_{x_i}^{x_{i+1}} \varphi(x)dxdt \\
 &- \int_0^T \sum_{i=0}^{N-1} R_i \left( \varphi(x_{i+1})\dot{x}_{i+1} - \varphi(x_i)\dot{x}_i \right)dt\,.
% &= \int_0^T \sum_{i=0}^{N-1} R_i (\dot{x}^{nl}_{i+1} - \dot{x}^{nl}_i) \intmed_{x_i}^{x_{i+1}} \varphi(t,\,x)dxdt - \int_0^T \sum_{i=0}^{N-1} R_i \left( \varphi(t,\,x_{i+1})\dot{x}^{nl}_{i+1} - \varphi(t,\,x_i)\dot{x}^{nl}_i \right)dt \\
% & + \int_0^T \sum_{i=0}^{N-1} R_i (\dot{x}^{d}_{i+1} - \dot{x}^{d}_i) \intmed_{x_i}^{x_{i+1}} \varphi(t,\,x)dxdt - \int_0^T \sum_{i=0}^{N-1} R_i \left( \varphi(t,\,x_{i+1})\dot{x}^{d}_{i+1} - \varphi(t,\,x_i)\dot{x}^{d}_i \right)dt,
\end{align*}
Recalling the expressions for $\dot{x}^{d}_i$ and $\dot{x}^{nl}_i$ in \eqref{part_diff} and \eqref{part_nonl} respectively, the convergence in~\eqref{weaksol} follows as soon as we show that $\lim_{N \to \infty} I^{d}_N = \lim_{N \to \infty} II^{nl}_N = 0$, where we have defined
\begin{align*}
I^{d}_N:=\quad&\int_0^T \sum_{i=0}^{N-1} R_i (\dot{x}^{d}_{i+1} - \dot{x}^{d}_i) \intmed_{x_i}^{x_{i+1}} \varphi(x)dxdt - \int_0^T \sum_{i=0}^{N-1} R_i \left( \varphi(x_{i+1})\dot{x}^{d}_{i+1} - \varphi(x_i)\dot{x}^{d}_i \right)dt \\
&\,+ \int_0^T \sum_{i=0}^{N-1}\phi(R_i)(\varphi_x(x_{i+1})-\varphi_x(x_i))dt, 
\end{align*}
and
\begin{align*}
II^{nl}_N:=\quad&\int_0^T \sum_{i=0}^{N-1} R_i (\dot{x}^{nl}_{i+1} - \dot{x}^{nl}_i) \intmed_{x_i}^{x_{i+1}} \varphi(x)dxdt - \int_0^T \sum_{i=0}^{N-1} R_i \left( \varphi(x_{i+1})\dot{x}^{nl}_{i+1} - \varphi(x_i)\dot{x}^{nl}_i \right)dt \\
&\,- \int_0^T \sum_{i=0}^{N-1}\int_{x_i}^{x_{i+1}} R_i v(R_i)\varphi_x(x)\left[ \sum_{j=0}^{N-1} R_j (K(x-x_{j+1})-K(x-x_j)) \right]dxdt,
\end{align*}
We first focus on the limit of the diffusive term. If we expand $\varphi$ at the first order with respect to $x_{i+1}$ and $x_i$ respectively, for some $\alpha_i \in (x_i,x),\,\beta_i \in (x,x_{i+1})$ we get
\begin{align*}
I^d_N &= \int_0^T \sum_{i=0}^{N-1} R_i \left[\dot{x}^d_{i+1} \left(\intmed_{x_i}^{x_{i+1}} \varphi(x)dx - \varphi(x_{i+1})\right) - \dot{x}^d_i \left(\intmed_{x_i}^{x_{i+1}} \varphi(x)dx - \varphi(x_i)\right) \right] \\
&\quad + \int_0^T \sum_{i=0}^{N-1}\phi(R_i)(\varphi_x(x_{i+1})-\varphi_x(x_i))dt \\
&= \int_0^T \sum_{i=0}^{N-1} (\phi(R_i)-\phi(R_{i+1}))\left(-\frac{\varphi_x(x_{i+1})}{2} - \frac{\varphi_{xx}(\beta_i)}{6}(x_{i+1}-x_i)\right) \\
&\quad- \int_0^T \sum_{i=0}^{N-1}(\phi(R_{i-1})-\phi(R_i))\left(\frac{\varphi_x(x_i)}{2} - \frac{\varphi_{xx}(\alpha_i)}{6}(x_{i+1}-x_i)\right) \\
&\quad+ \int_0^T \sum_{i=0}^{N-1}\phi(R_i)(\varphi_x(x_{i+1})-\varphi_x(x_i))dt \\
%&= \int_0^T \sum_{i=0}^{N-1} \phi(R_i)\left(\frac{\varphi_x(x_{i+1})-\varphi_x(x_i)}{2} - (x_{i+1}-x_i)\frac{\varphi_{xx}(\beta_i)-\varphi_{xx}(\alpha_i)}{6}\right) \\
%&\quad- \int_0^T \sum_{i=1}^{N}\phi(R_i) \left(\frac{\varphi_x(x_i)}{2} - (x_i - x_{i-1})\frac{\varphi_{xx}(\beta_i)}{6}\right) \\
%&\quad+ \int_0^T \sum_{i=0}^{N-2}\phi(R_i) \left(\frac{\varphi_x(x_{i+1})}{2} - (x_{i+2} - x_{i+1})\frac{\varphi_{xx}(\alpha_i)}{6}\right) \\
%&= B.T. + \frac{1}{6}\int_0^T \sum_{i=1}^{N-1} \phi(R_i) \left(\varphi_{xx}(\beta_{i-1})(x_i - x_{i-1}) - \varphi_{xx}(\beta_i)(x_{i+1}-x_i)\right) \\
%&\quad+\frac{1}{6}\int_0^T \sum_{i=0}^{N-2} \phi(R_i) \left(\varphi_{xx}(\alpha_{i})(x_{i+1} - x_i) - \varphi_{xx}(\alpha_{i+1})(x_{i+2}-x_{i+1})\right) \\
&= B.T. + \int_0^T \sum_{i=1}^{N-1} (x_{i+1}-x_i)\left[(\phi(R_{i+1})-\phi(R_i))\frac{\varphi_{xx}(\beta_i)}{6} + (\phi(R_i)-\phi(R_{i-1}))\frac{\varphi_{xx}(\alpha_i)}{6}\right],
\end{align*}
where $B.T.$ involves the terms related to $x_0$ and $x_{N}$. Since $\varphi_x(t,x_0)=\varphi_x(t,x_N)=0$ for all $t$ because of the zero velocity condition and $\varphi$ has compact support in $[0,\ell]$, by using the upper bound of~\eqref{minmax} and~\eqref{boundphi}, $B.T.$ can be easily estimated as follows:
\begin{equation}\label{B.T.}
|B.T.| \leq \| \varphi_{xx}\|_{L^{\infty}}\frac{C(\phi,T)}{N}\,.
\end{equation}
While, thanks to,~\eqref{TV} and the upper bound of~\eqref{minmax}, we deduce
\begin{align}\label{resto}
\notag &\int_0^T \sum_{i=1}^{N-1} \left|(x_{i+1}-x_i)\left[(\phi(R_{i+1})-\phi(R_i))\frac{\varphi_{xx}(\beta_i)}{6} + (\phi(R_i)-\phi(R_{i-1}))\frac{\varphi_{xx}(\alpha_i)}{6}\right] \right|dt\\
&\leq \frac{4}{6} \|\varphi_{xx}\|_{L^{\infty}} \int_0^T\sum_{i=0}^{N-1} |x_{i+1} -x_i||\phi(R_{i+1})-\phi(R_i)|dt \leq \|\varphi_{xx}\|_{L^{\infty}}  \frac{C(T,\bar{\rho},\phi)}{N},
\end{align}
then~\eqref{B.T.} and~\eqref{resto} together imply that 
\begin{equation}\label{diff}
\lim_{N\to\infty} I^{d}_N  = 0.
\end{equation}
Let us focus now on the nonlocal term $II^{nl}_N $. For future use we compute 
\begin{equation}\label{differenzaderivatenl}
\int_0^T \sum_{i=0}^{N-1} \big(\dot{x}^{nl}_{i+1}(t) - \dot{x}^{nl}_i(t)\big) dt \leq 2Lip[v]LT\,TV[\bar{\rho}] + v_{max}L\ell\,. 
\end{equation}
Notice that, expaning $\varphi$ to the first order, one can find $\vartheta_i \in (x_i,x_{i+1})$ such that 
\begin{align*}
II^{nl}_N &+ \int_0^T \sum_{i=0}^{N-1}\int_{x_i}^{x_{i+1}} R_i v(R_i)\varphi_x(x)\left[ \sum_{j=0}^{N-1} R_j (K(x-x_{j+1})-K(x-x_j)) \right]dxdt = \\
&= \int_0^T \sum_{i=0}^{N-1} R_i (\dot{x}^{nl}_{i+1} - \dot{x}^{nl}_i) \left[ \intmed_{x_i}^{x_{i+1}} \varphi(x)dx - \varphi(x_{i+1}) \right] -  \int_0^T \sum_{i=0}^{N-1} R_i \dot{x}^{nl}_i \left(\varphi(x_{i+1})-\varphi(x_i)\right)dt\\
%&= \int_0^T \sum_{i=0}^{N-1} R_i \left[(\dot{x}^{nl}_{i+1} - \dot{x}^{nl}_i)\intmed_{x_i}^{x_{i+1}} (\varphi_x(x_{i+1}) (x-x_{i+1}) + \varphi_{xx}(\vartheta_i)  (x-x_{i+1})^2)dx - \dot{x}^{nl}_i \int_{x_i}^{x_{i+1}} \varphi_x (x)dx\right]dt \\
%&= \int_0^T \sum_{i=0}^{N-1} R_i\left[(\dot{x}^{nl}_{i+1} - \dot{x}^{nl}_i)\left(  \varphi_x(x_{i+1})\frac{(x_{i+1}-x_i)}{2} +  \varphi_{xx}(\vartheta_i) \frac{(x_{i+1}-x_i)^2}{3}\right) - \dot{x}^{nl}_i \int_{x_i}^{x_{i+1}} \varphi_x (x)dx\right] dt  \\
&= \int_0^T  \sum_{i=0}^{N-1} \left[ (\dot{x}^{nl}_{i+1} - \dot{x}^{nl}_i) \left(\frac{\varphi_x(x_{i+1})}{2N} +\varphi_{xx}(\vartheta_i)\frac{ (x_{i+1} - x_i)}{3N} \right) - R_i \dot{x}^{nl}_i \int_{x_i}^{x_{i+1}} \varphi_x (x)dx \right]dt\,.
\end{align*}
Then we can rewrite $II^{nl}_N$ as the sum of three terms 
\begin{align*}
&A^{nl}_N:= \int_0^T  \sum_{i=0}^{N-1} (\dot{x}^{nl}_{i+1} - \dot{x}^{nl}_i) \frac{\varphi_x(x_{i+1})}{2N}dt\,, \\
&B^{nl}_N:= \int_0^T \sum_{i=0}^{N-1} (\dot{x}^{nl}_{i+1} - \dot{x}^{nl}_i) \frac{\varphi_{xx}(\vartheta_i)(x_{i+1} - x_i)}{3N} dt\,, \\
&C^{nl}_N:= \int_0^T \sum_{i=0}^{N-1} R_i \int_{x_i}^{x_{i+1}} \varphi_x(x) [\dot{x}^{nl}_i + \frac{v(R_i)}{N} \sum_{j=0}^{N-1} R_j \big( K(x-x_j) - K(x-x_{j+1})\big)]dxdt\,,
\end{align*}
so the remaining part of proof consists in showing that $|A^{nl}_N|,\,|B^{nl}_N|,\,|C^{nl}_N|$ vanishes as $N \to\infty$. 
The first two terms are immediate, indeed
\[ |A^{nl}_N| \leq \frac{\| \varphi_x \|_{L^{\infty}}}{2N}  \int_0^T \sum_{i=0}^{N-1}|\dot{x}^{nl}_{i+1} - \dot{x}^{nl}_i| dt   \quad \mbox{ and } \quad |B^{nl}_N| \leq \frac{\ell\| \varphi_{xx} \|_{L^{\infty}} }{2N}  \int_0^T \sum_{i=0}^{N-1} |\dot{x}^{nl}_{i+1} - \dot{x}^{nl}_i| dt, \]
which, thanks to~\eqref{differenzaderivatenl} immediately imply
\begin{equation}\label{AB_N}
|A^{nl}_N| + |B^{nl}_N| \leq \frac{C(T,K,v,\bar{\rho})}{N}, 
\end{equation}
On the other hand, the term $C^{nl}_N$, is less straightforward. Expanding $K$ at the first order, for every $j=0,\ldots,N-1$ we can find some $\beta_j \in (x-x_{j+1},\,x-x_j)$ such that
%\[ K(x-x_{j+1})  =  K(x - x_j) + K'(x-x_j) (x_j - x_{j+1}) + K''(\beta_j)(x_j - x_{j+1})^2, \]
%thus $C^{nl}_N$ becomes
\begin{align*}
C^{nl}_N = & \int_0^T \sum_{i=0}^{N-1}R_i \int_{x_i}^{x_{i+1}} \varphi_x(x) \frac{v(R_i)}{N}\left[K'(x-x_i) + \sum_{j \neq i} K'(x - x_j)  - \sum_{j \neq i} K'(x_i-x_j) \right] dxdt \\
& + \int_0^T  \sum_{i=1}^{N-1}R_i \int_{x_i}^{x_{i+1}} \varphi_x(x)  \left[ \frac{v(R_{i}) - v(R_{i-1})}{N} \sum_{j<i} K'(x_i - x_j) \right] dxdt \\
& + \int_0^N  \sum_{i=0}^{N-1}R_i \int_{x_i}^{x_{i+1}} \varphi_x(x) \frac{v(R_i)}{N^2} \sum_{j=0}^{N-1} K''(\beta_j)(x_{j+1}-x_j) dxdt\,.
\end{align*}
Then we want to show that each of the three terms has order $1/N$. For the first one we use the Lipschitz regularity of $K'$ and the uniform bound on $\varphi$ and $v$, indeed
\begin{align}\label{C_N1}
\notag&\int_0^T \sum_{i=0}^{N-1}R_i \int_{x_i}^{x_{i+1}} \varphi_x(t,\,x) \frac{v(R_i)}{N}\left[ K'(x-x_i) + \sum_{j \neq i} K'(x - x_j)  - \sum_{j \neq i} K'(x_i-x_j) \right] dxdt \\
&\leq L \|v\|_{L^{\infty}} \| \varphi_x \|_{L^{\infty}} \int_0^T \sum_{i=0}^{N-1}R_i \int_{x_i}^{x_{i+1}} v(R_i) \sum_{j>i} |x-x_i| dxdt \leq \| \varphi_x \|_{L^{\infty}} \frac{C(K,T,v,\ell)}{N}.
\end{align} 
Recalling~\eqref{TV}, the Lipschitz regularity of $v$, and the uniform bound of $K'$ we can estimate the second term as follows
\begin{align}\label{C_N2}
\notag \int_0^T  &\sum_{i=1}^{N-1}R_i \int_{x_i}^{x_{i+1}} \varphi_x(t,\,x)  \left[ \frac{v(R_{i}) - v(R_{i-1})}{N} \sum_{j<i} K'(x_i - x_j) \right] dxdt \\
%\notag &\leq L\,Lip[v]\| \varphi_x \|_{L^{\infty}}  \int_0^T \sum_{i=1}^{N-1} |R_i - R_{i-1}|R_i(x_{i+1} - x_i)dt \\
&= L\,Lip[v] \| \varphi_x \|_{L^{\infty}} \frac{1}{N} \int_0^T \sum_{i=1}^{N-1} |R_i - R_{i-1}| dt  \leq  \| \varphi_x \|_{L^{\infty}([0,\,T]\times \mathbb{R})} \frac{C(K,T,v,\bar{\rho})}{N}\,.
\end{align}
We take care now of the term involving $K''$. In this case, we can even achieve an upper bound of order $N^{-2}$. Recalling that $L$ is so that $\|K''\|_{L^{\infty}([-\ell,\ell])} < L$, we get
\begin{align}\label{C_N3}
\notag\int_0^T  &\sum_{i=0}^{N-1}R_i \int_{x_i}^{x_{i+1}} \varphi_x(t,\,x) \frac{v(R_i)}{N} \sum_{j=0}^{N-1} K''(\beta_j)(x_{j+1}-x_j) dxdt \\
&\leq L\, \|v\|_{L^{\infty}} \|\varphi_x\|_{L^{\infty}} \frac{1}{N^2} \int_0^T \sum_{j=0}^{N-1}  (x_{j+1}-x_j) dxdt \leq \|\varphi_x\|_{L^{\infty}([0,\,T]\times \mathbb{R})} \frac{C(K,v,T,\ell)}{N^2}\,.
\end{align}
Summarizing, estimates~\eqref{C_N1},~\eqref{C_N2},~\eqref{C_N3} imply that 
\begin{equation*}
|C^{nl}_N| \leq \frac{C(\varphi,\,K,\,\ell,\,T,v)}{N},
\end{equation*} 
and, together with estimates~\eqref{AB_N} and~\eqref{diff}, that $\lim_{N\to\infty} II^{nl}_N = 0$ and the validity of~\eqref{weaksol}. 
\end{proof}

\section{Numerical simulations}
This last section of the paper is devoted to the numerical study of equation \eqref{main} in order to compare the different results that came from different possible choices of the diffusion function. More precisely, we implement the procedure introduced in Section 2, first solving numerically the ODE system in \eqref{ODES} and then taking the piecewise constant reconstruction \eqref{rhoN} as a numerical approximation for the corresponding PDE. Most of the examples below highlight the competition between diffusive phenomena and aggregation phenomena. Assumptions in (Ker) allow us to analyze only \emph{smooth enough} kernels and it is not too restrictive to consider in all the simulations $K$ as an attractive Gaussian potential
\begin{equation}\label{agg}
 K(x)=\mathcal{K}\left(1-e^{-|x|^2}\right), \quad \mathcal{K}>0.
\end{equation}
The velocity function $v$ in the mobility term will be taken in the form 
\[
 v(\rho)=(1-\rho)_+,
\]
that is a usual choice in the applications. We will show different behaviors in the following depending on the choice of the diffusion function $\phi$.
In \figurename~\ref{fig:confronto} we compare final configurations for the aggregation-diffusion problem in which the diffusion of porous medium type depends on a parameter $\epsilon$
\begin{equation}\label{qd}
 \phi_{PM}(\rho)=\frac{\epsilon}{2}\rho^2,
\end{equation}
while the aggregation term is the Gaussian \eqref{agg} with $\mathcal{K}=1$. We set final time $T=1$, $N=300$ particles and we consider as initial configuration a uniform distributed density $\bar{\rho}(x)=0.7$ for $x\in\left[0,1\right]$. As shown in \figurename~\ref{fig:confronto}, the final configurations for $\epsilon=1,\, 0.1,\, 0.05$, are (numerically) stable, i.e. diffusion and aggregation phenomena compensate each other.    

%%% Figura 1
\begin{figure}[!ht]
\begin{center}
\begin{minipage}[c]{.7\textwidth}
\includegraphics[width=1\textwidth]{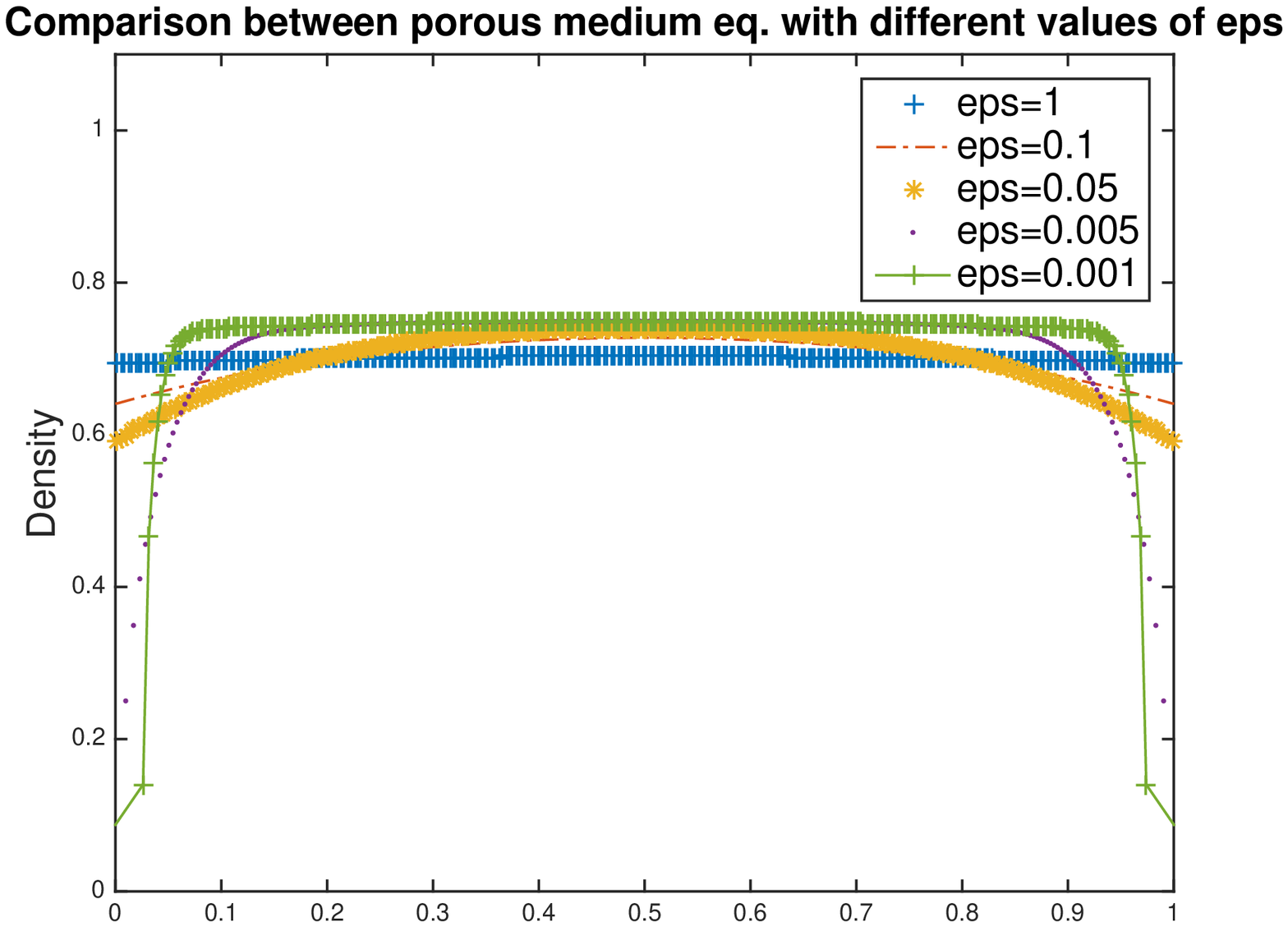}
\end{minipage}
\end{center}
\caption{Final configurations to aggregation-diffusion equations with $\phi_{PM}$ given by \eqref{qd} for different values of the diffusion coefficient $\epsilon$.}
\label{fig:confronto}
\end{figure}

Even if the diffusion coefficients are very small, and hence the aggregation effect is leading, the vacuum regions are forbidden. This is evident in \figurename~\ref{fig:novacuum} where we plot the stable density related to the case $\epsilon=0.001$.
%%% Figura 2
\begin{figure}[!ht]
\begin{center}
\begin{minipage}[c]{.48\textwidth}
\includegraphics[width=1.1\textwidth]{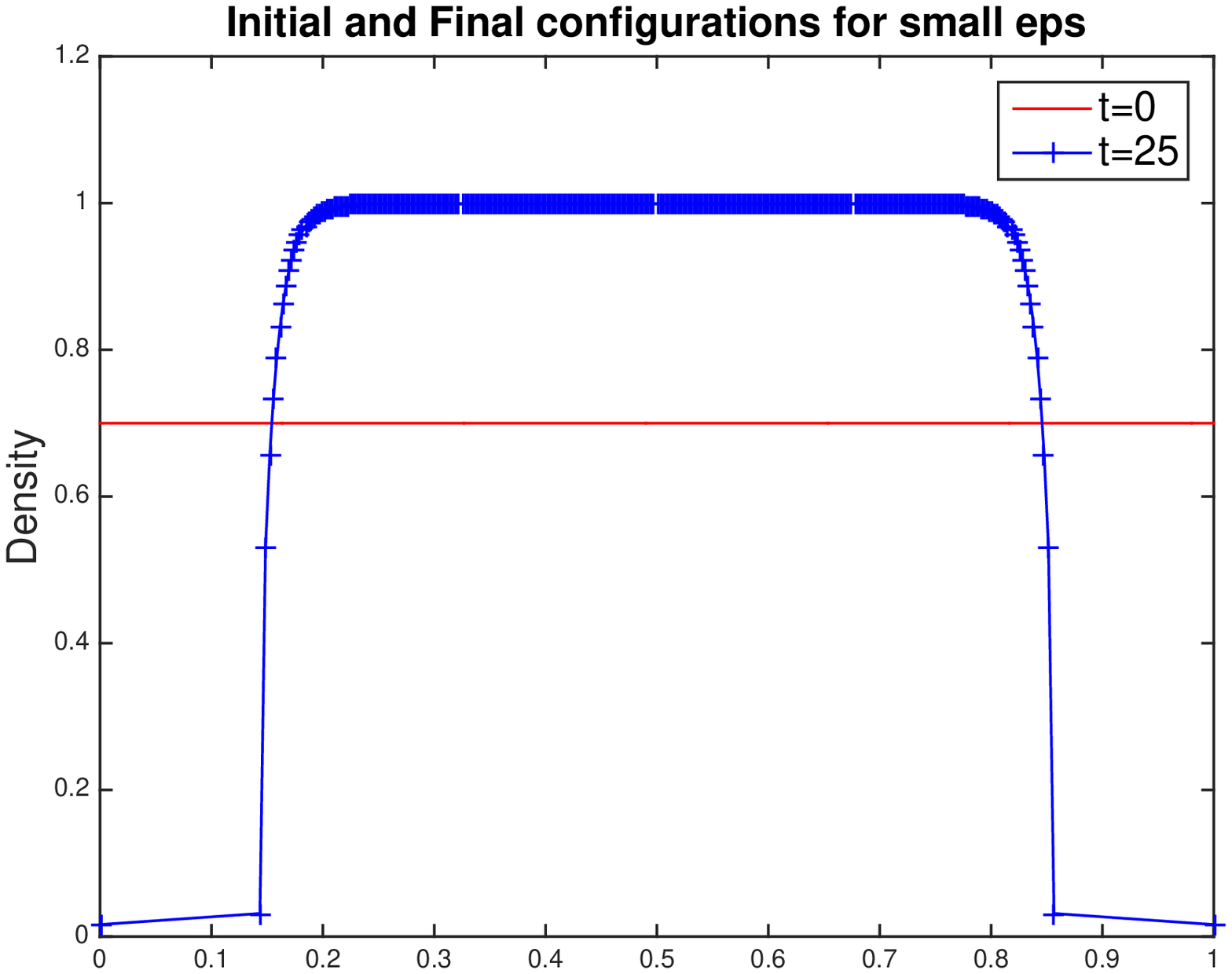}
\end{minipage}
\begin{minipage}[c]{.48\textwidth}
\includegraphics[width=1.1\textwidth]{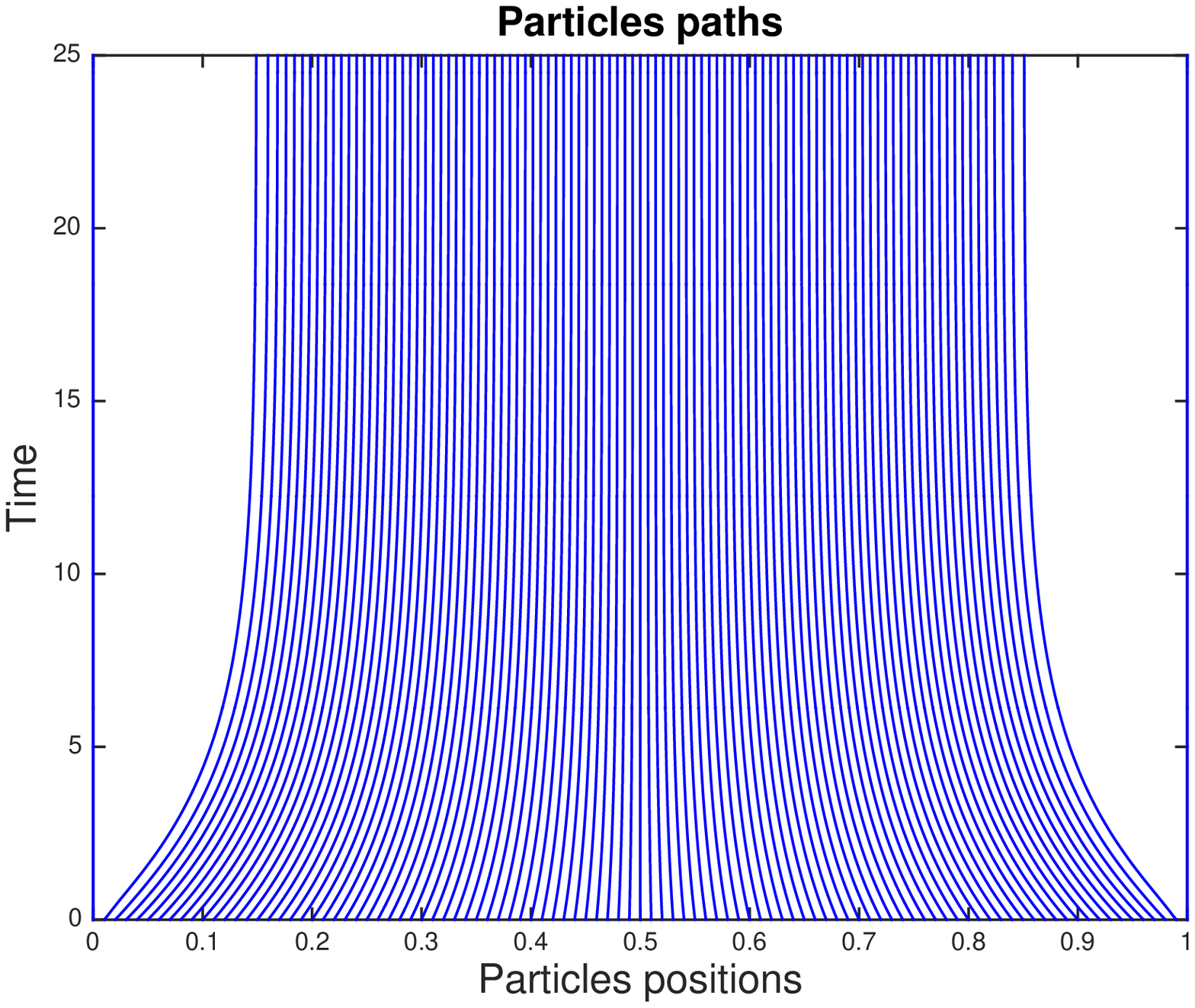}
\end{minipage}
\end{center}
\caption{Top. For $\epsilon=0.001$ in \eqref{qd} we are in the \emph{aggregation-dominated} regime. The degeneracy of $v$ stops the aggregation when the value $\rho=1$ is reached.}
\label{fig:novacuum}
\end{figure}

In \figurename~\ref{fig:confronto_2} we test the same initial condition with a diffusion function that presents a two-point degeneracy. We set
\begin{equation}\label{veldif}
\phi_{TP}(\rho)=\eps\int_0^\rho \zeta(1-\zeta)\zeta^{m-2}d\zeta=\frac{\eps}{m}\rho^m-
\frac{\eps}{m+1}\rho^{m+1},
\end{equation}
that is exactly a diffusion of porous medium type with nonlinear mobility, see \eqref{eq:diffusione}. The quadratic case plotted in \figurename~\ref{fig:confronto_2} shows the weaker effect of this diffusion functions. Indeed, if one takes as a reference the case $\epsilon=1$, the plot in \figurename~\ref{fig:confronto_2} is more concave than the respective one in \figurename~\ref{fig:confronto}.
%% Figura 3
\begin{figure}[!ht]
\begin{center}
\begin{minipage}[c]{.7\textwidth}
\includegraphics[width=1\textwidth]{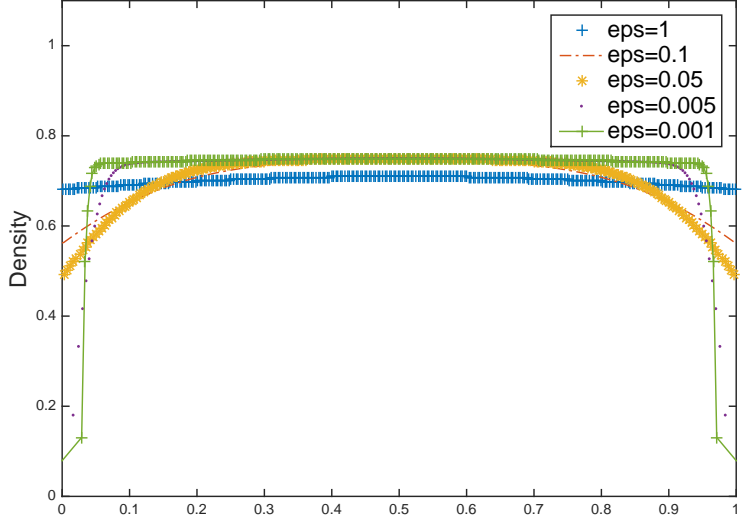}
\end{minipage}
\end{center}
\caption{Final configurations to aggregation-diffusion equations with $\phi_{TP}$ given by \eqref{veldif} for different values of the diffusion coefficient $\epsilon$.}
\label{fig:confronto_2}
\end{figure}

Let us present now some simulations in the \emph{strongly degenerate} diffusion regime. We consider 
\begin{equation}\label{strongdif}
\phi_{SD}(\rho)=\begin{cases}
 \frac{\epsilon}{2}\rho^2 &\quad  \rho\in\left[0,\frac25\right),\\
 \frac{2\epsilon}{25} &\quad  \rho\in\left[\frac25,\frac35\right),\\
 \frac{2\epsilon}{25}+ \frac{\epsilon}{2}(\rho-\frac35)^2 &\quad  \rho\geq \frac35.
\end{cases}
\end{equation}
Since we are in a smooth setting, in the time intervals in which $\phi_{SD}'\equiv 0$, the evolution is driven only by the aggregation term. This may result in formation of discontinuities in the density profile, similarly to the flux-saturated degenerate parabolic equations, see \cite{CKR}. In the left column of \figurename~\ref{fig:strong} we plot the evolution of an initial datum that is constantly equal to a value in the range of degeneracy of $\phi_{SD}'$. In the right column of \figurename~\ref{fig:strong}, instead, we consider a two steps initial datum with only one of the two values is critical for $\phi_{SD}'$.

 %% Figura 4
\begin{figure}[!ht]
\begin{center}
\begin{minipage}[c]{.48\textwidth}
\includegraphics[width=1.1\textwidth]{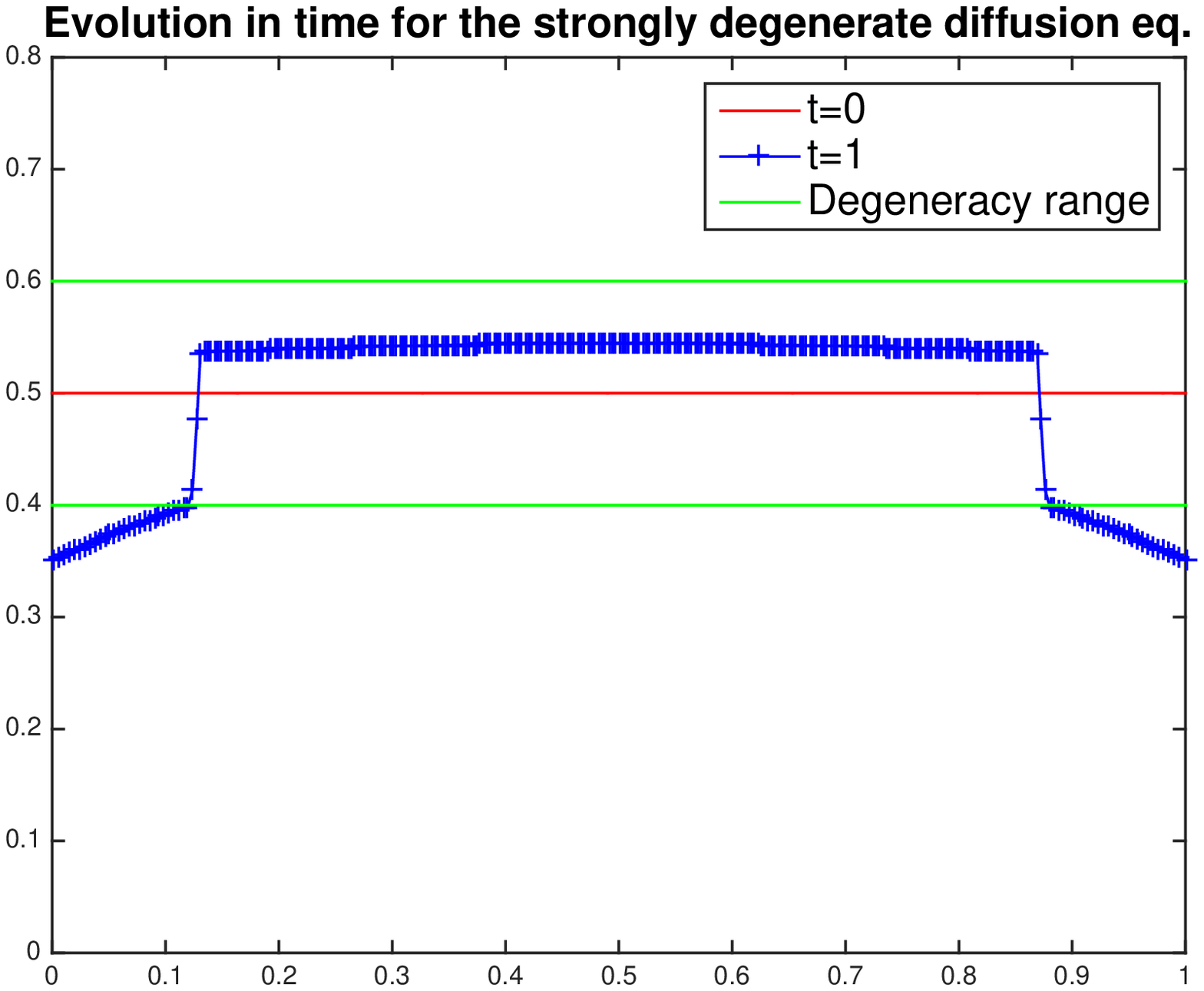}
\end{minipage}
\begin{minipage}[c]{.48\textwidth}
\includegraphics[width=1.1\textwidth]{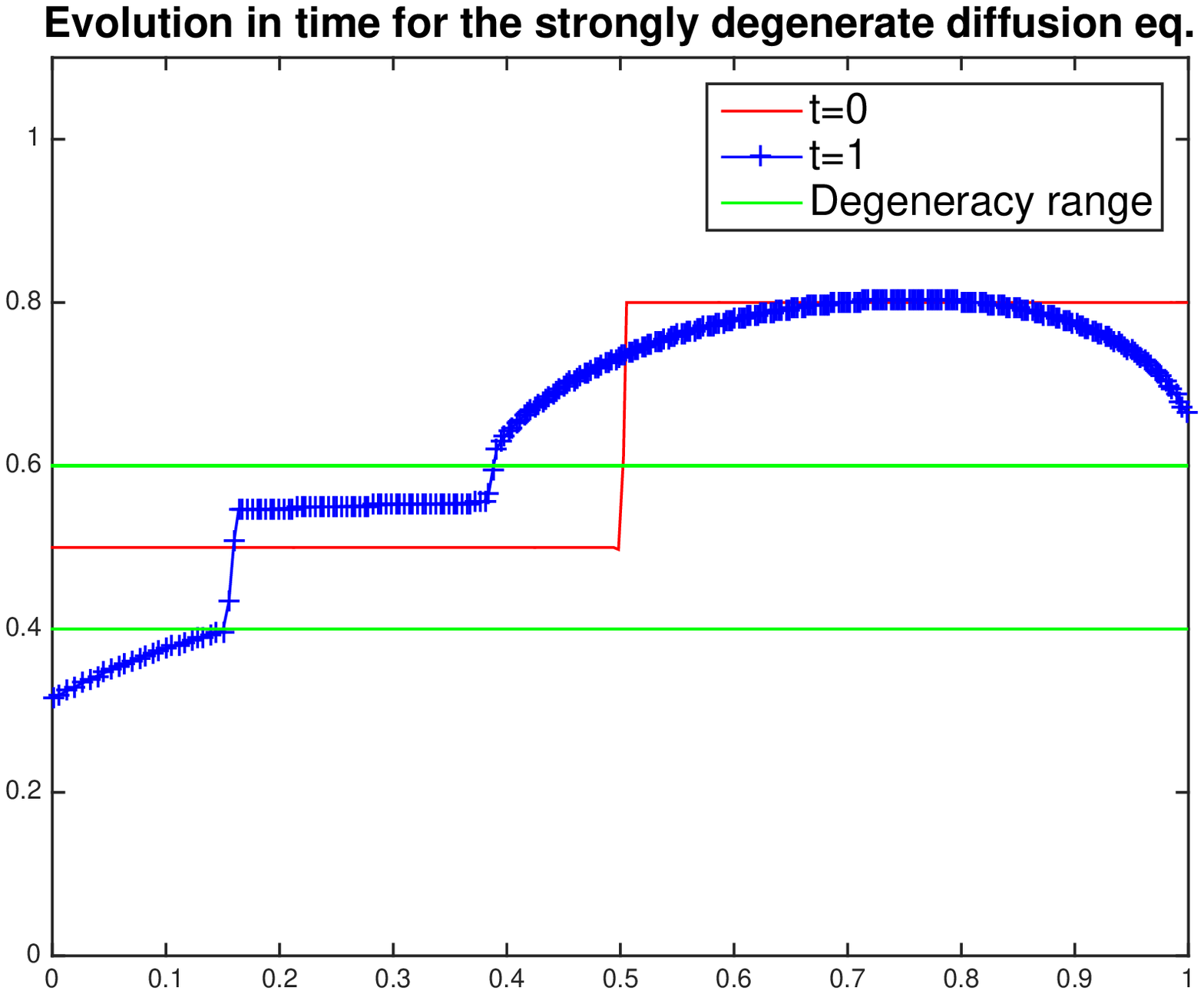}
\end{minipage}
\begin{minipage}[c]{.48\textwidth}
\includegraphics[width=1.1\textwidth]{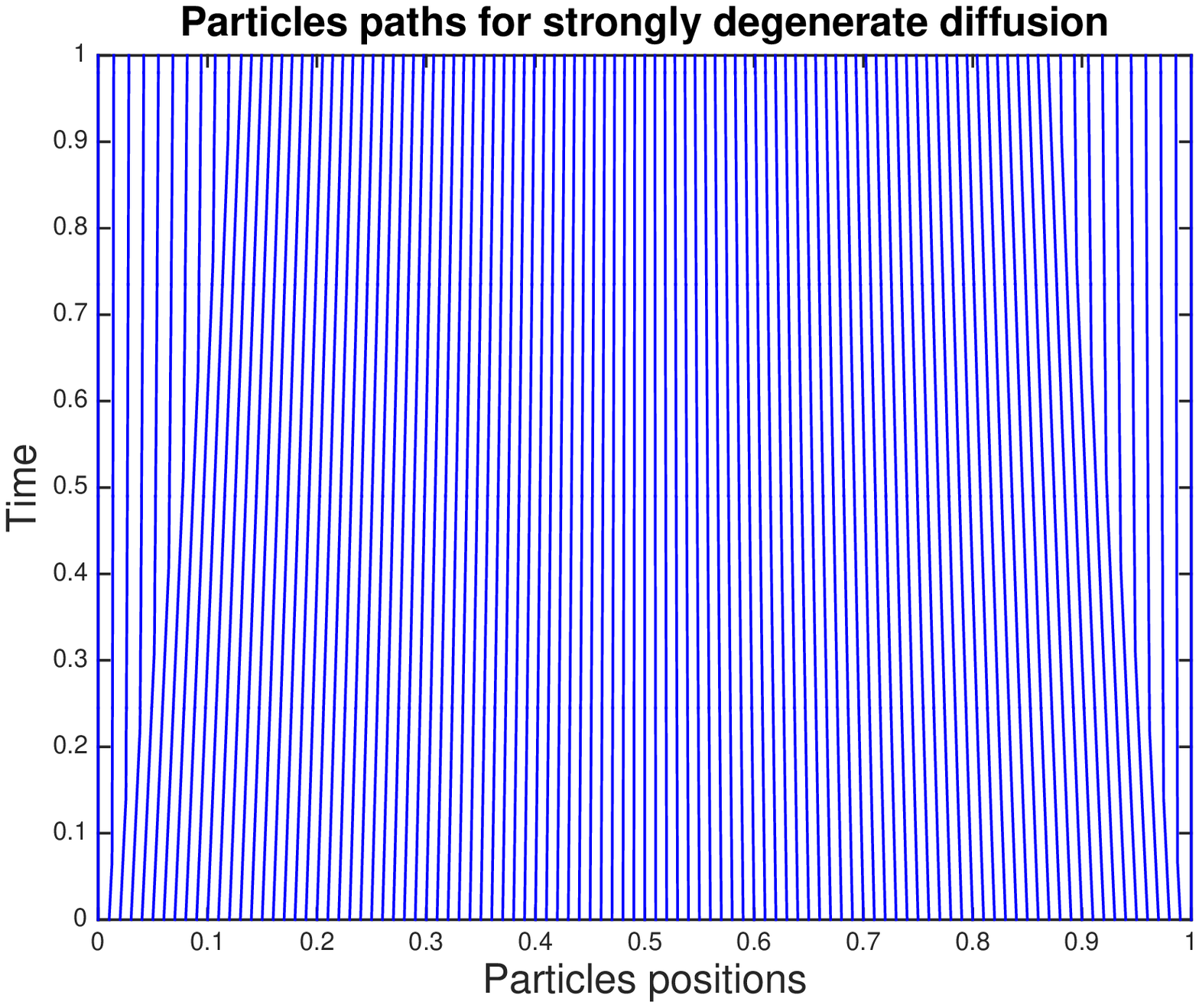}
\end{minipage}
\begin{minipage}[c]{.48\textwidth}
\includegraphics[width=1.1\textwidth]{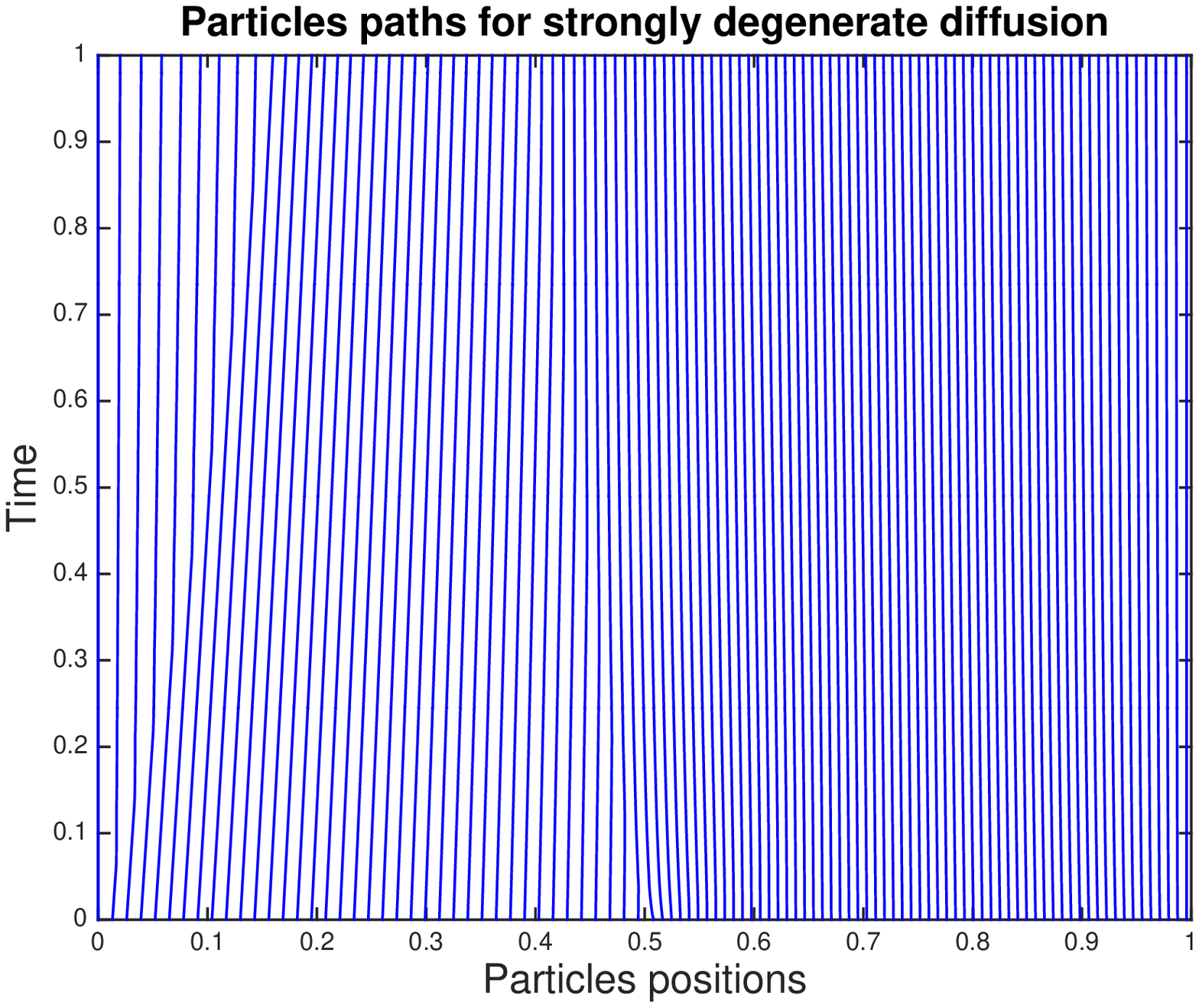}
\end{minipage}
\end{center}
\caption{The horizontal (green) lines denotes the range of degeneracy of $\phi_{SD}'$.\\
\emph{Left}. Initial and final configurations (top) and particles trajectories (bottom) for constant initial value with strongly degenerate diffusion. \\
\emph{Right}. Initial and final configurations and particles trajectories for the two steps initial value with strongly degenerate diffusion.}
\label{fig:strong}
\end{figure}

In the last example, see \figurename~\ref{fig:confronto_3},  we compare the evolutions of the oscillating initial datum 
\begin{equation}\label{seno}
 \bar{\rho}(x)=\frac12\left(\cos(4\pi x)+1\right), \quad x\in\left[0,2\right],
\end{equation}
with respect to the same aggregation potential \eqref{agg} and the different diffusion functions $\phi_{PM}$, $\phi_{TP}$ and $\phi_{SD}$.

 %% Figura 5
\begin{figure}[!ht]
\begin{center}
\begin{minipage}[c]{.48\textwidth}
\includegraphics[width=1.1\textwidth]{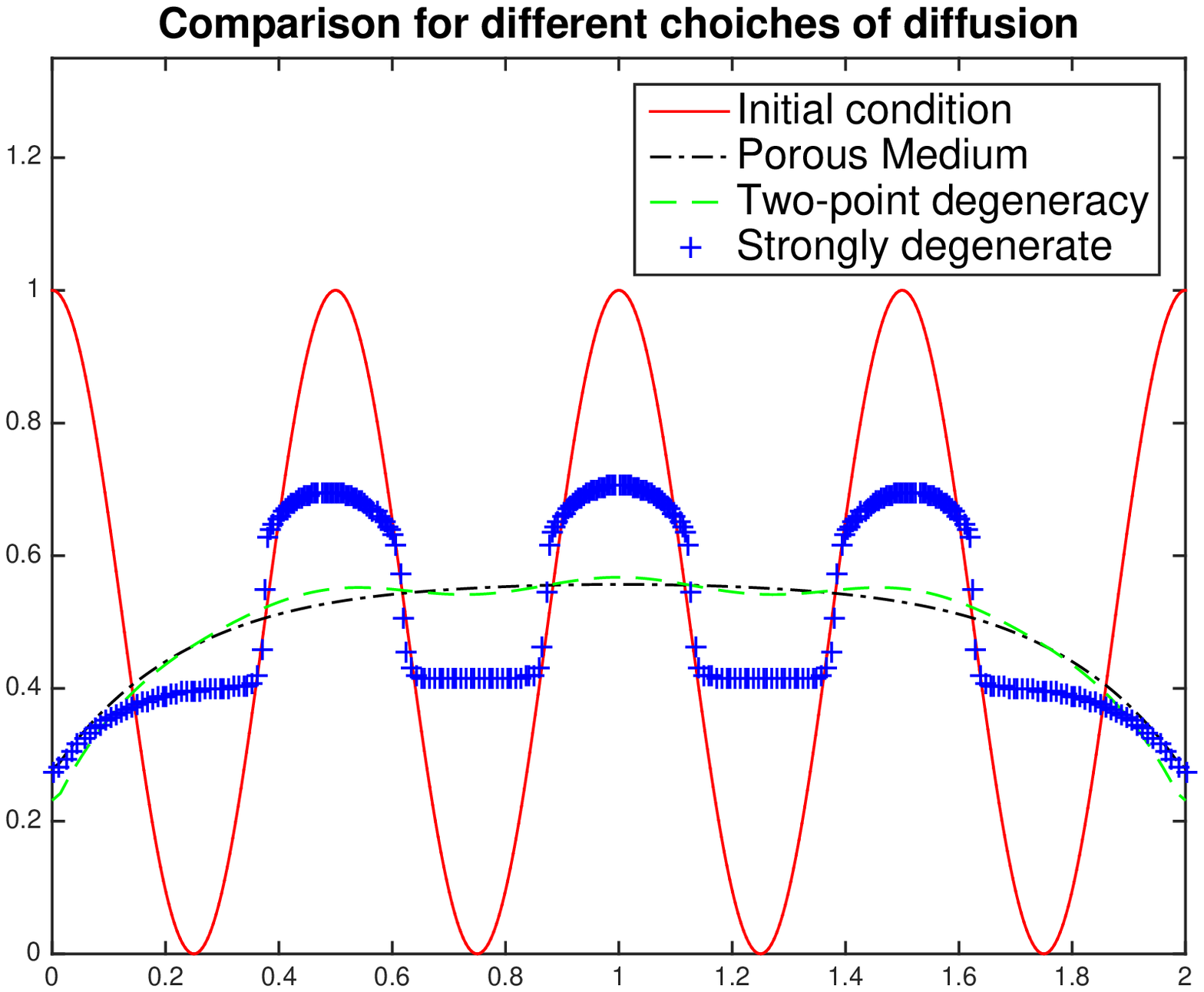}
\end{minipage}
\begin{minipage}[c]{.48\textwidth}
\includegraphics[width=1.1\textwidth]{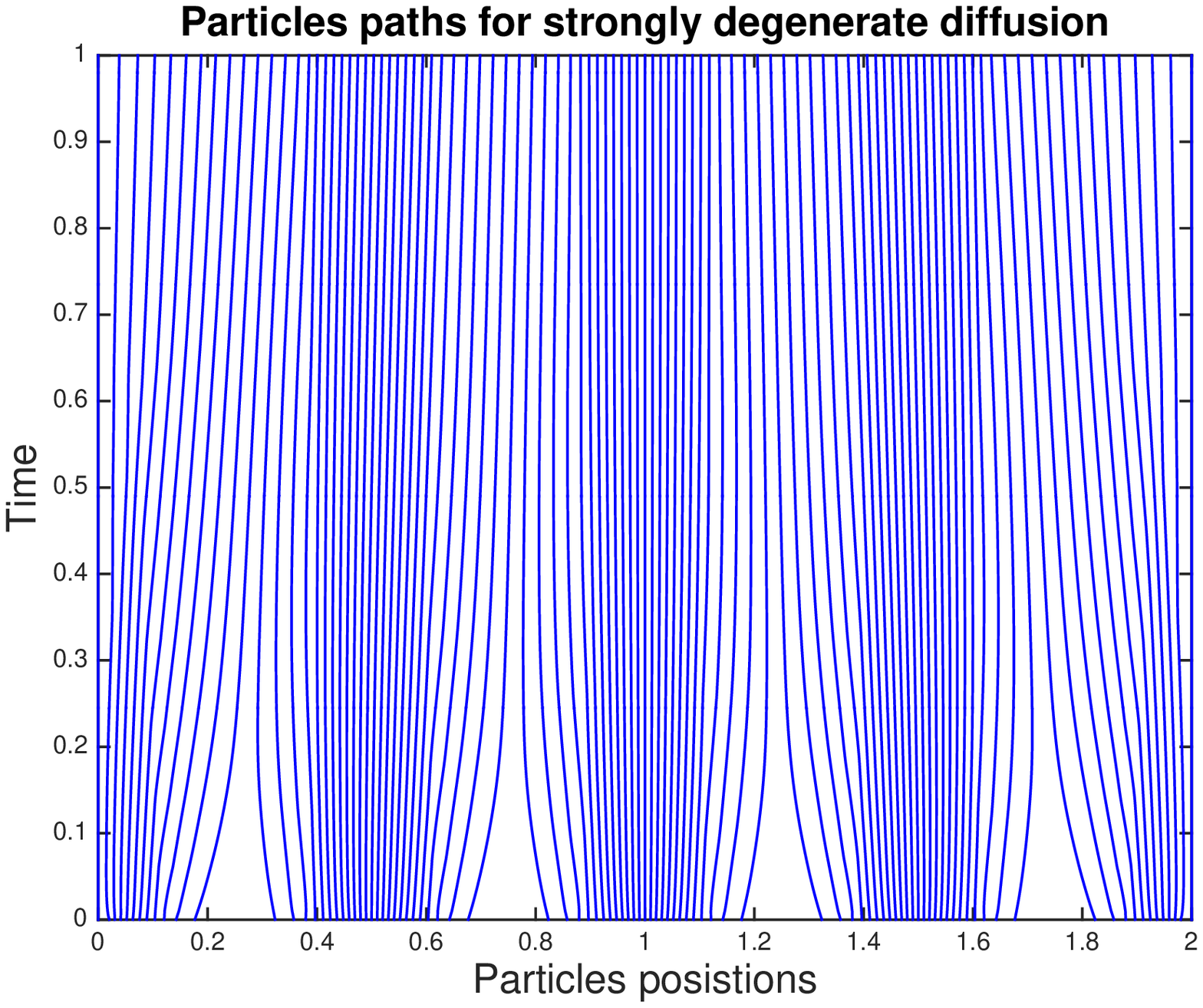}
\end{minipage}
\end{center}
\caption{Comparison between the final configurations for Gaussian potential \eqref{agg}, diffusion  $\phi_{PM}$, $\phi_{TP}$ and $\phi_{SD}$ and initial datum as in \eqref{seno}. The particle trajectories on the right correspond to the strongly degenerate case.}
\label{fig:confronto_3}
\end{figure}

\section{Conclusions and Perspectives}
The main purpose of this paper is to use a \emph{deterministic} (ODEs) particle approximation method to construct solutions to a fairly wide class of aggregation-diffusion equations with nonlinear mobility, which are largely used in several contexts in population biology (see the Introduction). We stress that the presence of nonlinear mobility is new compared to several results available in the literature (see for instance \cite{MatSol}). This issue, together with the presence of the nonlocal transport term, suggests us to adapt to our case the strategy developed in \cite{DFR} for scalar conservation laws. As a main result we prove that a suitable piecewise  constant density reconstruction converges strongly to weak solutions of the aggregation-diffusion equation in the sense of Definition \ref{weaksolutiondfn}.
We highlight that our approach is alternative to other possible approaches, such as the kinetic or probabilistic methods mentioned in the Introduction. In particular, the diffusion term is described by means of the \emph{deterministic} discrete osmotic velocity as in the seminal paper by Russo in the 90s. This type of description is able to catch the degenerate diffusion case, that can be useful in the modeling of biological phenomena. In this sense we rely on the approach by \cite{GT1}.  An interesting aspect of this deterministic approach is the capability of being designed as a powerful numerical scheme: by solving the system of ODEs \eqref{ODES} and reconstructing the density as in \eqref{rhoN}, we are able to catch interesting phenomena, such as the formation of discontinuities. This feature is of great use in the class of equations we consider here, in which the degeneracy of the diffusion term in some nontrivial intervals allows for the formation of shocks as we show in our simulations.

We plan to extend our approach to systems with many species, see \cite{fagioli} as an example in predator-prey interactions, and to the multidimensional case in the spirit of \cite{CHPW}.

%%%%%%
\section*{Acknowledgements}
The authors would like to thank Marco Di Francesco for pointing out the problem and for the several fruitful discussions and comments. The first author wishes to thank Giovanni Russo for the helpful suggestions at the very early stage of this paper. 
The authors acknowledge support from the EU-funded Erasmus Mundus programme `MathMods - Mathematical models in engineering: theory, methods, and applications' at the University of L'Aquila, from the Italian GNAMPA mini-project `Analisi di modelli matematici della fisica,
della biologia e delle scienze sociali', and from the local fund of the University
of L’Aquila `DP-LAND (Deterministic Particles for Local And Nonlocal Dynamics).

%%%%%%

\end{document}